\documentclass[12pt,reqno,a4paper]{amsart}

\usepackage[top=4cm,bottom=3.5cm,left=3cm,right=3cm]{geometry}

\usepackage[utf8]{inputenc}
\usepackage[T1]{fontenc}
\usepackage[main=british,french,german,latin]{babel}

\usepackage{amsmath,amsthm,amssymb,mathtools}
\usepackage{mathrsfs}
\usepackage{stmaryrd}

\numberwithin{equation}{section}

\usepackage{microtype}
\usepackage{xcolor}
\usepackage{hyperref}

\hypersetup{
  colorlinks=true,
  hypertexnames=false,
  linkcolor=blue!60!black,
  citecolor=blue!60!black,
  urlcolor=blue!60!black,
  pdfauthor={},
  pdftitle={}
}

\usepackage{enumitem}

\theoremstyle{plain}
\newtheorem{theorem}{Theorem}[section]
\newtheorem{proposition}[theorem]{Proposition}

\newtheorem{lemma}[theorem]{Lemma}

\newtheorem{definition}[theorem]{Definition}

\theoremstyle{definition}

\newtheorem{remark}[theorem]{Remark}

\DeclareMathOperator{\e}{e}

\DeclareMathOperator{\im}{Im}

\newcommand{\N}{\mathbb{N}}
\newcommand{\Z}{\mathbb{Z}}

\newcommand{\R}{\mathbb{R}}
\newcommand{\C}{\mathbb{C}}
\newcommand{\T}{\mathbb{T}}

\renewcommand{\epsilon}{\varepsilon}

\makeatletter
\@namedef{subjclassname@2020}{\textup{2020} Mathematics Subject Classification}
\makeatother

\begin{document}

\title[Modified scattering for NLS on Diophantine waveguides]{Modified scattering for the cubic Schr{\"o}dinger equation on Diophantine waveguides}

\author{Nicolas Camps}

\address{Univ Rennes, CNRS, IRMAR - UMR 6625, 
F-35000 Rennes, France
}

\email{nicolas.camps@univ-rennes.fr}


\author{Gigliola Staffilani}
 \address{Department of Mathematics,  Massachusetts Institute of Technology,  77 Massachusetts Ave,  Cambridge,  MA 02139-4307 USA.
}

\email{gigliola@math.mit.edu.}


\date{\today}

\begin{abstract}
We consider the cubic Schrödinger equation posed on a product space subject to a generic Diophantine condition. Our analysis shows that the small-amplitude solutions undergo modified scattering to an effective dynamics governed by interactions that induce no growth of high-order Sobolev norms. This is in sharp contrast with the infinite energy cascade scenario observed by Hani--Pausader--Tzvetkov--Visciglia in the absence of  Diophantine conditions. 
 \end{abstract} 

\ \vskip -1cm  \hrule \vskip 1cm \vspace{-8pt}
 \maketitle 
{ \textwidth=4cm \hrule}

\setcounter{tocdepth}{1}
\tableofcontents

\section{Introduction}

\subsection{Motivations} The present work aims to test the persistence under domain perturbation of turbulent mechanisms observed in the long-time dynamics of certain solutions of nonlinear dispersive equations. On compact domains, the interplay between weaker dispersion and nonlinear resonances can lead to energy transfers from low to high frequency scales of oscillations as time evolves. These forward cascade mechanisms manifest themselves through the growth of the high-order Sobolev norms.  Bourgain's original formulation in \cite{Bou-conj-00} addresses the potential explosion of Sobolev norms in infinite time, known as \emph{the infinite energy cascade}. Despite considerable efforts to elucidate this conjecture, our understanding of the infinite-time behavior of nonlinear dispersive equations on compact domains is still incomplete.   

\medskip

One of the very few results constructing infinite energy cascades is due to Hani--Pausader--Tzvetkov--Visciglia \cite{HPTV15}. They consider the cubic defocusing Schrödinger equation posed on the product space $\R\times\T^{d}$, with $\T^{d}=\R^{d}/\Z^{d}$:
\begin{equation}
\tag{NLS}
\label{eq:nls_0}
i\partial_{t}U+\partial_{x}^{2}U + \Delta_{y}U = |U|^{2}U\,,\quad (x,y)\in\R\times\T^{d}\,.
\end{equation}
The dispersion of the Schr\"odinger free evolution along the non-compact direction $\mathbb{R}$ serves to dampen and scatter the non-resonant interactions, reducing the effective dynamics to a system governed by the four-waves resonant set of the cubic Schr\"odinger equation posed on the compact direction~$\mathbb{T}^{d}$. This enables the construction of small amplitude initial data that ignite a forward energy cascade (\cite[Corollary 1.4]{HPTV15}):
\medskip

\noindent{
\it For all $d\in\{2,3,4\}$ and $s>1$ large enough, there exists $U(0)$ with small amplitude~\footnote{The precise statement actually requires the smallness of a strong norm, with extra derivatives and space decay in the $\R$-direction} in $H^{s}(\R\times\T^{d})$ such that the corresponding solution $U$ is global and satisfies 
\[
\limsup_{t\to+\infty}\|U(t)\|_{H^{s}(\R\times\T^{d})}=+\infty\,.
\]

}
\medskip

In this work, we allow for a more general dispersion relation in the compact direction $\T^{d}$
\begin{equation}
\tag{A-NLS}
\label{eq:nls}
i\partial_{t}U + \partial_{x}^{2}U + \operatorname{div}(\mathrm{A}\nabla_{y})U = |U|^{2}U\,,\quad (x,y)\in\R\times\T^{d}\,.
\end{equation}
The symmetric positive definite constant matrix $\mathrm{A}$ encodes the dispersion relation, and we will denote
\[
\Delta_{\mathrm{A}}:=\partial_{x}^{2}+\operatorname{div}(\mathrm{A}\nabla_{y})\,.
\]
Changing the dispersion via $\mathrm{A}$ corresponds to rescaling $\mathbb{T}^{d}$ through a change of variables. For instance when $\mathrm{A}=\operatorname{Id}$, we retrieve the case of the square torus $\mathbb{T}^{d}$, while rectangular tori correspond to diagonal matrices $\mathrm{A}$. We will focus on generic matrices satisfying an explicit Diophantine condition presented in Definition \ref{def:ad}. They correspond to generic non-rectangular tori. 

\medskip

A natural question is whether the turbulent mechanisms observed in the case of the square torus \cite{HPTV15} persist under generic and non-periodic boundary conditions. 

\medskip

Our main results, stated in Theorem \ref{thm:main} and Theorem \ref{thm:ms}, address this question. When \eqref{eq:nls} is posed on a waveguide with a generic Diophantine property on~$\mathrm{A}$, we establish that there is no amplification of the Sobolev norms over infinite time for the solutions associated with small amplitude initial data. Hence, the asymptotic behavior of \eqref{eq:nls} on the product space~$\R\times\T^{d}$ with $d\geq2$ is very different when~$\T^{d}$ is replaced by  Diophantine tori.
This sharp dichotomy, which manifests clearly at the analytical level, is also a warning against using periodic boundary conditions to model dispersive equations on bounded domains. What we prove here is that any infinitesimal perturbation of a periodic domain may result in a dramatically different behavior of the energy spectrum of the solutions. 

\subsection{Context}  In this subsection we recall some previous related works and further discuss the main motivations.

\subsubsection{Growth of Sobolev norms for nonlinear Schrödinger equations on compact domains and waveguides.} The existence of infinite cascades realized via the growth of Sobolev norms has drawn considerable interest, and a survey of all the existing results goes much beyond the scope of this paper, hence here we recall only the works that are strictly related to it. 
 
\medskip

A line of research within the study of the infinite energy cascade consists in proving polynomial upper bounds for the growth of Sobolev norms~\footnote{The Cauchy theory and the conservation of the energy give exponential bounds} to quantify the rate at which energy transfers to high frequencies. This was done originally by Bourgain~\cite{Bou96IMRN}, and several other works have appeared in recent years
\cite{Staffilani1997, PTV17, Soinger2011, Soinger2011-1, Soinger2012}. On rectangular Diophantine tori \cite{DG19} improved the polynomial bounds, and \cite{D19} obtained polynomial bounds for the energy-critical quintic Schrödinger equation (no analogous result is known for the square torus). These results suggest that energy transfers, if they exist, take longer to occur when the boundary conditions are Diophantine rather than periodic. 

\medskip

As mentioned before, much less is known about the lower bounds and the infinite growth of the Sobolev norms. The infinite energy cascade evidenced in \cite{HPTV15} is a consequence of the analysis of the Toy Model, a subsystem of the four-waves resonant set. This model was introduced in \cite{CKSTT10} in the study of the cubic Schrödinger equation on the square torus $\T^{2}$ (see also \cite{CF12}). In that case, using a stability lemma, one could only partially transfer the growth dynamics of the Toy Model and prove a norm-amplification result in finite time:
\medskip

\noindent{\it {Given $s>1$, a small $\epsilon>0$ and a large $R>0$}, there exist $T>0$ and $u(0)\in H^{s}(\T^{2})$ such that the solution $u$ to \eqref{eq:nls_0} posed on $\T^{2}$ satisfies
\[
\|u(0)\|_{H^{s}(\T^{2})}\leq \epsilon\,,\quad \|u(T)\|_{H^{s}(\T^{2})}\geq R\,.
\]

}
Along these lines we also recall the work in \cite{Hani2019, Haus2016, Haus2015, GK15}. 
  On compact domains, transposing the growth of the resonant dynamics to the full equation requires a stability lemma, which restricts such norm amplification results to finite time-scales. A remarkable exception occurs for completely integrable models, such as the cubic Szeg\H{o} equation \cite{GG10} --- which arises as the resonant system of the half-wave equation --- whose integrable structure provides other mechanisms yielding growth of Sobolev norms in infinite time.

\subsubsection{Infinite energy cascade on waveguides} On the waveguide manifold, the dispersion along the $\R$ direction allows one to fully realize the growth predicted by the Toy Model over infinite time. The strategy developed in \cite{HPTV15} (also explained in \cite{HPTV14}) consists in propagating a \emph{weak norm} based on the energy space which is preserved by the effective resonant system, while allowing the stronger norms to slightly grow in time. This yields both modified scattering and the construction of modified wave operators, through which the turbulent dynamics predicted by the Toy Model can be transferred to the full equation \eqref{eq:nls_0}. The time-oscillations in the non-resonant interactions are easily~\footnote{Indeed, the resonance function for NLS on $\T^{d}$ (periodic boundary conditions) takes integer values. If it does not vanish then it takes integer values greater than 1, unlike in the present work where small divisor problems arise.} converted into decay via a Poincaré--Dulac normal form, while non-stationary phase arguments control space-oscillations, in the spirit of \cite{KP11}.   
\medskip

Subsequent works, such as \cite{Xu17,Che24},  adopt the same strategy for weakly dispersive equations, resulting in infinite energy cascades~\footnote{In these two works, however, the Cauchy theory at low regularity is not well understood and poses considerable problems. They could achieve the construction of modified wave operators, which is sufficient to prove infinite energy cascade, without proving the modified scattering result.}. Another notable example is \cite{DR19}, which extracts interesting dynamics with a beating effect from the resonant system of a coupled system of nonlinear Schrödinger equations.
\medskip

In the spirit of the present work, \cite{GPT16} considers the cubic Schrödinger equation perturbed by some convolution potential and prove that there is no norm amplification. For this model, the exact resonances are trivial, but the quasi-resonances lead to small divisor issues. Nevertheless, for generic choices of convolution potential they are controlled by the third largest frequency and can be easily handled. 
\medskip

In our setting (namely Diophantine tori without any convolution potential), however, the admissibility condition \eqref{eq:ad} provides a control on the small divisors by the highest frequency. This small divisor issue necessitates new approaches that will be discussed in Section \ref{sec:pr} below. 
\medskip

The authors of \cite{WY22} also consider the NLS equation posed on irrational waveguides. Following the strategy proposed in \cite{HPTV15}, they prove that some \emph{lower} Sobolev norms remain bounded in infinite time. By accepting such a loss of derivatives they are able to circumvent the aforementioned small divisor issues.  Their results do not isolate any effective dynamics, and do not exclude infinite cascade or norm amplification.

In contrast, we demonstrate modified scattering for small amplitude solutions towards an effective dynamics as detailed in Theorem \ref{thm:ms}. By analyzing the effective dynamics, we deduce that there is no norm amplification at the regularity of the initial data, as stated in Theorem \ref{thm:main}. 
\medskip

To conclude this subsection, we emphasize that on $\R^{n}\times\T^{d}$ with $n\geq2$ there is scattering \cite{TV12,TV16}. Indeed, if $n\geq2$ then all the interactions are attenuated by the stronger dispersion on $\R^{n}$ .  

\subsubsection{Nonlinear Schrödinger equations posed on Diophantine tori: longer time-scales?}

By considering irrational tori instead of the square torus, dispersion is enhanced and the refocusing time needed for the waves to explore the geometry is also longer. This was quantified by \cite{DGG17, DGGMR22} who obtained longer time versions of Strichartz's estimates on Diophantine rectangular and non-rectangular tori than in \cite{BD15}.

\medskip

Moreover, for the Schrödinger equation in \eqref{eq:nls_0} on irrational tori some exact resonant interactions are eliminated, whereas quasi-resonant interactions emerge: the resonance function is not valued in $\Z$ and can take arbitrarily small values without vanishing. In~\cite{GG22}, it was proved that the quasi-resonant interactions can also lead to turbulent mechanisms by extending the norm amplification result of \cite{CKSTT10} for small amplitude initial data to the case of some rectangular irrational tori in dimension $d=2$, after exponential time-scale (the Diophantine approximation condition for the length of the torus is detailed in \cite[paragraph 2.1]{GG22}). This result indicates that on irrational tori quasi-resonant interactions are dynamically relevant. 
\medskip

Nevertheless, the turbulent mechanisms are weaker in this setting. The first author, together with Bernier, recently developed in~\cite{BC24} a normal form approach establishing stability over
polynomial time-scales~\footnote{Namely over time-scales $\|u(0)\|_{H^{s}(\T^{d})}^{-r}$ for any prescribed order $r\geq2$.} for the cubic Schrödinger equation posed on flat tori satisfying the Diophantine condition \eqref{eq:ad}. In \cite{SW20}, the second author with Wilson proved that on rectangular irrational tori the resonant system is much smaller than on $\T^{2}$ and showed that the energy transfers observed in \cite{CF12} out of the Toy Model when finitely many modes are excited do not longer occur. Moreover in \cite{HPSW21} the second author with Hrabski, Pan and Wilson also gave some numerical evidence that
the high Sobolev norms should not grow as rapidly as the ones on a square torus. For more numerical results one should also consult \cite{Colliander-num, hrabski2020effect}. 
Finally, \cite{CG20} considers generic Diophantine tori in the context of wave kinetic theory to bypass the time-scales dictated by the resonant interactions. 
\subsection{Set up and main results}
\label{sec:thm}
We consider the equation \eqref{eq:nls} on the waveguide manifold with the following explicit Diophantine condition on the coefficients of $\mathrm{A}$. 
\begin{definition}
\label{def:ad}
The matrix $\mathrm{A}$ is admissible with parameters $\tau>0$ and~$c>0$ if for all $(a,b)\in(\Z^{d}\setminus\{0\})^{2}$, 
\begin{equation}
\label{eq:ad}
|a^{\intercal}\mathrm{A}b|\geq\frac{c}{|a|^{\tau}|b|^{\tau}}\,.
\end{equation} 
\end{definition}
\begin{remark}
The admissibility condition~\eqref{eq:ad} is closely related to classical Diophantine lower bounds for linear forms. In particular, an application of the Borel-Cantelli lemma (see also the convergence case of Khinchin~\cite{Khin}'s theorem) asserts that for Lebesgue-almost every $(\beta_i)_{2\le i\le d}\in[1,2]^{d-1}$ there exists $C>0$ such that for all
$k\in\Z^{d}\setminus\{0\}$,
\[
\bigl|k_{1}+\beta_{2}k_{2}+\cdots+\beta_{d}k_{d}\bigr|
\ge 
C\prod_{i=2}^{d}(1+|k_i|)^{-1}\bigl(\log(2+|k_i|)\bigr)^{-2}.
\]
This formulation appears for instance in~\cite[D2]{DGGMR22}. Such a lower bound implies a bound of the form
\[
|k_{1}+\beta_{2}k_{2}+\cdots+\beta_{d}k_{d}|
\gtrsim |k|^{-\tau}
\]
for any $\tau>d-1$. In this sense, admissibility with parameter $\tau>\frac{d(d+1)}{2}-1$ holds generically.
\end{remark}

The operator $-\operatorname{div}(\mathrm{A}\nabla)$ has discrete spectrum, with eigenvalues indexed by the points~$n\in\Z^{d}$:
\[
-\operatorname{div}(\mathrm{A}\nabla)\e^{in\cdot y} = \lambda_{n}^{2}\e^{in\cdot y}\,,\quad \text{with}\quad \lambda_{n}^{2}:= n^{\intercal}\mathrm{A}n = \sum_{i,j=1}^{d}n_{i}n_{j}\mathrm{A}_{i,j}\,.
\]
Since $\mathrm{A}$ is symmetric and positive definite,  there exist $c, C>0$ such that for all~$n\in\Z^{d}$,
\begin{equation}
\label{eq:co}
c|n|^{2} \leq \lambda_{n}^{2} \leq C|n|^{2}\,,\quad \text{where we recall that}\quad |n|^{2}:=\sum_{i=1}^{d}n_{i}^{2}\,.
\end{equation} 
The four-waves resonance function  is
\begin{align*}
\Omega(\vec{n}) &= \lambda_{n_{1}}^{2} - \lambda_{n_{2}}^{2}+\lambda_{n_{3}}^{2} - \lambda_{n}^{2}
\,.
\intertext{Under the zero-momentum condition $n=n_{1}-n_{2}+n_{3}$ it factorizes to}
\Omega(\vec{n}) &= 2(n_{1}-n_{2})^{\intercal}\mathrm{A}(n_{2}-n_{3})\,.
\end{align*}
This factorization motivates Definition \ref{def:ad} of the Diophantine admissibility condition for the matrix $\mathrm{A}$. Firstly, it eliminates the exact resonances, as proved in Lemma \ref{lem:4w} below. Secondly, it introduces quasi-resonant interactions but provides a quantitative polynomial decay of the resonance function. This gives hope to exploit the time-oscillations in  $H^{s}(\T^{d})$ despite the presence of small divisors, which appear in the normal forms method used to convert time-oscillations into decay in time. 
\medskip

Rectangular tori are not admissible. If the vector $(\mathrm{A}_{i,j})_{i,j}$ is Diophantine then $\mathrm{A}$ is admissible with a parameter~$\tau$ corresponding to the irrationality number of~$\mathrm{A}$, but this is not an equivalence. We refer to the introduction of~\cite{BC24} for more details on the admissibility condition.
\medskip

The precise definitions of the norms are given in Section \ref{sec:norm} below. At this stage, we only stress that the $S$-norm, which appears in the statement of our main theorems, controls $s$ derivatives in the compact direction $\T^{d}$ and $s+\sigma$ derivatives in the~$\R$ direction, for an arbitrarily small $\sigma>0$.  Other spaces  that we use in the statements are the Sobolev space $H^{s}(\R\times\T^{d})$
and the space $h^{s}(\T^{d})$, which is equivalent to the standard Sobolev space $H^{s}(\T^{d})$, as will be proved in Section \ref{sec:norm}.
We also anticipate that for $\tau>0$ and~$c(d)\in(0,2]$ given in  Lemma \ref{lem:sep},  the regularity threshold needed for our results  is
\begin{equation}
\label{eq:par}
s_{0}(\tau,d):=\frac{d}{2}+\frac{4\tau}{c(d)}\,.
\end{equation}
We are now ready to state the main theorems. 
\begin{theorem}
\label{thm:main} Let $d\geq2$ and $\mathrm{A}$ be admissible with parameter $\tau$. Given~$s>s_{0}(\tau,d)$, there exist $C>0$ and $\epsilon_{\ast}>0$ such that for all $\epsilon\in(0,\epsilon_{\ast})$, if 
\[
\|U(0)\|_{S} \leq \epsilon\,,
\]
then the solution $U$ to \eqref{eq:nls} with Cauchy data $U(0)$ exists globally in $C(\R;S)$. Moreover, for all $t\in\R$,
\begin{equation}
\label{eq:r1}
	\|U(t)\|_{H^{s}(\R\times\T^{d})}
	\leq C\epsilon\,,\quad
	\|U(t)\|_{L^{\infty}_x h^{s}_y(\R\times\T^{d})}
	\leq C(1+|t|)^{-\frac{1}{2}}\epsilon\,.
\end{equation}
\end{theorem}
In particular~\eqref{eq:r1} shows that the solution retains the dispersive decay rate~$t^{-\frac{1}{2}}$ inherited from the unbounded direction.  The proof of Theorem \ref{thm:main} is based on the dynamics of an effective system presented below in~\eqref{eq:eff} and ~\eqref{eq:s-eff}, to which the small amplitude solutions undergo modified scattering.  The effective system does not only contain the resonant interactions (which are trivial when $\mathrm{A}$ is admissible), but also some specific quasi-resonant interactions described in~\eqref{eq:Q}.  
\begin{theorem}[Modified scattering] 
\label{thm:ms}
Under the same assumptions as in Theorem~\ref{thm:main}, we have that for all $U(0)$ with 
\[
\|U(0)\|_{S}\leq\epsilon\,,
\]
there exists a global solution $G\in C([1,\infty);H^{s}(\R\times\T^{d}))$ to the system \eqref{eq:s-eff} on $[1,\infty)$ such that
\[
\sup_{t\geq1}\|G(t)\|_{Z}+\|G(1)\|_{H^{s}(\R\times\T^{d})}\lesssim\epsilon
\]
and
\[
\lim_{t\to+\infty}\|\e^{-it\Delta_{\mathrm{A}}}U(t)-G(t)\|_{H^{s}(\R\times\T^{d})} = 0\,.
\]
\end{theorem}
We make some remarks.
\begin{remark} Infinite energy cascades are excluded, in sharp contrast to the case of the square torus studied in \cite{HPTV15}. To the best of the authors' knowledge, Theorem~\ref{thm:main} is the first result that shows that  the dynamics in infinite time is completely different on Diophantine tori from that  on the square torus. This indicates that the growth mechanism established in \cite{HPTV15} depends on the periodic boundary conditions and is somehow unstable.  
\end{remark}
\begin{remark}The $Z$-norm, that will be defined in  Section \ref{sec:norm} and used for the proofs of the theorems stated above, is the strongest norm preserved by the effective system. Unlike the case of the square torus \cite{HPTV15} this norm is not based on the energy space but on $H^{s}(\T^{d})$, for $s>\frac{d}{2}$. As a result, we no longer have low regularity issues and our results cover all the dimensions $d\geq2$. This also simplifies the analysis compared to~\cite{HPTV15}.
\end{remark}
\begin{remark} The case when $d=1$, which is also covered by the proof, is less interesting since there is no quasi-resonant or resonant interactions: there is modified scattering to the (trivial) resonant system \eqref{eq:R0}, as in the case with the convolution potential~\cite{GPT16}.
\end{remark}
\begin{remark}The regularity threshold $s>s_{0}(\tau,d)=\frac{d}{2}+\frac{4\tau}{c(d)}$ depends on the Diophantine properties of the dispersion relation (through the parameter $\tau$) and on the dimension  $d$ (through the constant $c(d)$ related to the frequency separation property given in Lemma \ref{lem:sep}). Since $c(d)\leq 2$ we have 
\[
s_{0}(\tau,d)\geq\frac{d}{2}+2\tau\,.
\]
One can expect to save some of the $\frac{d}{2}$-loss by using Strichartz's estimates from~\cite{BD15}. However, it is not clear how to avoid the loss $2\tau$ due to the quasi-resonant interactions. Since this loss is dominant~\footnote{To have a generic condition we need to suppose $\tau>\frac{d(d+1)}{2}-1$}, we prefer not to use the machinery of Strichartz's estimates. In the unbounded direction $\R$, we are able to make the loss of regularity arbitrarily small.
\end{remark}

\subsection{Idea of the proof}
\label{sec:pr}

The overall strategy is inspired from the analysis developed in \cite{HPTV15} and also in \cite{KP11}, with two important differences. 
\medskip

\noindent$-$ Firstly, the analysis of the time non-resonant interactions is more complicated due to small divisor problems. For some quasi-resonant interactions, we can still perform a Poincaré--Dulac normal form reduction to obtain an integrable decay in time, thanks to a frequency separation property based on a cluster decomposition recalled in Lemma \ref{lem:sep}. 
The remaining quasi-resonant interactions, which cannot be treated with a normal form method, contribute to the effective system. These non-resonant contributions are precisely the~high~$\times$~low~$\to$~high frequency interactions, where the incoming high frequency and the outgoing frequency are in the same dyadic cluster.  We give a precise definition of the effective system in \eqref{eq:Q}.

\medskip

\noindent$-$ Secondly, we show that these effective interactions do not trigger any growth of the Sobolev norms. Since the incoming high frequency and the outgoing frequency belong to the same dyadic cluster, we construct in Lemma \ref{lem:cons} an exact conserved quantity for the effective system at the level of the $h^{s}(\T^{d})$ norm.

Building on this stability, we propose a general strategy that leverages the dispersive decay together with the conservation of the Sobolev norms to prove modified scattering. Specifically, we make use of the natural dispersive estimate claimed in Theorem \ref{thm:main} to simplify the analysis of the space non-resonant interactions, and to give a concise self-contained proof. This allows us to prove the result in any dimension, and to have an arbitrarily small loss of derivatives in the $\R$-direction compared to the strong norm.

\subsection{Organization of the paper}
After introducing the Hamiltonian system in the transverse direction and exploiting the frequency separation property in Section \ref{sec:1}, we analyze nonlinear interactions in Section \ref{sec:non}, establishing bounds with suitable decay in time. Finally, in Section \ref{sec:a-priori}, we establish a priori bounds for the time-dependent norm, as presented in Proposition \ref{prop:a-priori}. These bounds yield the global existence result and the dispersive decay \eqref{eq:r1} asserted in Theorem \ref{thm:main}. Additionally, we prove Theorem~\ref{thm:ms} and deduce the boundedness of the $H^{s}$-norm, as also claimed in Theorem \ref{thm:main}.

\subsection{Notations}
For $F\in L^{2}(\R\times\T^{d})$, we define the (partial) Fourier transform:  for~$y\in\T^{d}$,
\begin{align*}
\xi\in\R\quad \mapsto\quad \widehat{F}(\xi,y) 
	&= (2\pi)^{-1}\int_{\R}\e^{-i\xi x}F(x,y)\mathrm{d}x\,
\intertext{and, for $x\in\R$,}
n\in\Z^{d}\quad \mapsto\quad F_{n}(x) 
	&= (2\pi)^{-d}\int_{\T^{d}}\e^{-in\cdot y}F(x,y)\mathrm{d}y\,.
\end{align*}
The (full) spatial Fourier transform of $F$ is 
\begin{equation}
\label{eq:f-f}
(\xi,n)\in\R\times\Z^{d}\ \mapsto\ (\mathcal{F}F)(\xi,n) := (2\pi)^{-(d+1)}\int_{\R\times\T^{d}}\e^{-i(\xi x+n\cdot y)}F(x,y)\mathrm{d}x\mathrm{d}y\,.
\end{equation}
Following \cite{HPTV15}, we write for $(\xi,n)\in\R\times\Z^{d}$
\[
\widehat{F}_{n}(\xi) := (\mathcal{F}F)(\xi,n)\,,\quad
\text{so that}
\quad
\widehat{F}(\xi,y)=\sum_{n\in\Z^{d}}\widehat{F}_{n}(\xi)\e^{in\cdot y}\,.
\]
Given three incoming waves $(n_{1},n_{2},n_{3})\in (\Z^{d})^{3}$, we define the decreasing rearrangement of $(|n_{1}|,|n_{2}|,|n_{3}|)$ by
\[
n_{1}^{\ast}\geq n_{2}^{\ast} \geq n_{3}^{\ast}\,.
\]
We denote by $n=n_{1}-n_{2}+n_{3}$ the outgoing wave and by $\vec{n}=(n_{1},n_{2},n_{3},n)\in(\Z^{d})^{4}$ the four interacting waves. 

\subsection*{Acknowledgments}
The authors thank the anonymous referees for their careful reading of the manuscript and for their valuable suggestions.
During the preparation of this work N.C. benefited from the support of the European research Council (ERC) under the European Union’s Horizon 2020 research and innovation programme (Grant agreement 101097172 - GEOEDP), the Centre Henri Lebesgue ANR-11-LABX-0020-0, the region Pays de la Loire through the project MasCan, the ANR project KEN ANR-22-CE40-0016.
G.S. was supported by the NSF grant DMS-2052651 and the Simons Foundation Collaboration  Grant on Wave Turbulence. 

\section{Preparations}
\label{sec:1}

\subsection{Clustering of the eigenvalues}
\label{subsec:tori}

We recall the frequency separation property, which plays a key role in the present work. This property was proved by Bourgain \cite[Lemma 9.29]{Bou98} on $\T^{d}$ (see also \cite{Bou99-cluster,Del10}) when $\mathrm{A}=\operatorname{Id}$, and generalized by Berti and Maspero~\cite{BM19} for all $\mathrm{A}$ symmetric and positive definite. The approach was recently introduced in the setting of nonlinear PDEs by~\cite{BFM24}, and further developed in~\cite{BC24} to justify finite-dimensional approximations in Birkhoff normal form constructions. 

\begin{lemma}[Frequency separation, Theorem 2.1 in \cite{BM19}]
\label{lem:sep}
There exist~$c(d)\in(0,2]$ and $C(A,d)\geq2$, and a partition of $\Z^{d}$ into
disjoint clusters $(\mathscr{C}_{\alpha})_{\alpha\geq 0}$ satisfying the following:
\begin{enumerate}
\item The first cluster contains the origin $0\in\mathscr{C}_{0}$ and is bounded:
\[
\max_{n\in \mathscr{C}_{0}}|n|\leq C(A,d)\,.
\]
\item The other clusters are dyadic: for all $\alpha\geq1$,
\[
\max_{n\in\mathscr{C}_{\alpha}} |n|\leq 2\min_{n\in\mathscr{C}_{\alpha}} |n|\,.
\]

\item For all $(n_{1},n_{2})\in\mathscr{C}_{\alpha_{1}}\times\mathscr{C}_{\alpha_{2}}$ with $\alpha_{1}\neq\alpha_{2}$,
\begin{equation}
\label{eq:sep}
|n_{1}-n_{2}|  + |\lambda_{n_{1}}^{2}-\lambda_{n_{2}}^{2}| > (|n_{1}|+|n_{2}|)^{c(d)}\,.
\end{equation}
\end{enumerate}
\end{lemma}
The last condition is a separation property that plays a key role in the analysis, as it allows one to control part of the quasi-resonant interactions (see Lemma~\ref{lem:small-div}).
\subsection{Norms}\label{sec:norm} 
We fix the regularity $s>s_{0}$ (see \eqref{eq:par}) as in the statement of Theorem~\ref{thm:main}. In our analysis, the relevant Sobolev norms in the compact direction are based on the cluster decomposition. We prove in Lemma \ref{lem:cons} that they are preserved by the effective system defined in~\eqref{eq:s-eff}. 
\begin{definition}[Sobolev norms based on the clusters]
\label{def:sob}
We define 
\[
K_{0}:=1\,,\quad K_{\alpha}:=\min_{n\in\mathscr{C}_{\alpha}}|n|\quad \text{for}\ \alpha\geq1\,,
\]
We index the nonzero clusters so that the sequence $(K_{\alpha})_{\alpha\geq1}$ is nondecreasing.
and $\pi_{\alpha}$ is the orthogonal projector onto $\operatorname{span}\ \{e^{in\cdot y}\}_{n\in\mathscr{C}_{\alpha}}$: for a function $u\in \ell^{2}(\T^{d})$, 
\[
\pi_{\alpha}u(y) = \sum_{n\in\mathscr{C}_{\alpha}}u_{n}\e^{in\cdot y}\,,\quad u_{n} = \frac{1}{(2\pi)^{d}}\int_{\T^{d}}\e^{-in\cdot y}u(y)\mathrm{d}y\,.
\]
For $s\in\R$, we define the Sobolev norms for functions on $\T^{d}$ and on $\R\times\T^{d}$ by 
 \begin{align}
\label{eq:norm-t}
\|u\|_{h^{s}(\T^{d})} &:=
	 \big(
	\sum_{\alpha\geq0} K_{\alpha}^{2s}\|\pi_{\alpha}u\|_{\ell^{2}(\T^{d})}^{2}
	 \big)^{\frac{1}{2}}
	 =
	 \big(\sum_{\alpha\geq0}K_{\alpha}^{2s}\sum_{n\in\mathscr{C}_{\alpha}}|u_{n}|^{2}\big)^{\frac{1}{2}}\,,
	 \\
	 	\label{eq:H}
	\|U\|_{H_{x,y}^{s}(\R\times\T^{d})} &:= (\|U\|_{L_{x}^{2}h_{y}^{s}(\R\times\T^{d})}^{2}+\|U\|_{H_{x}^{s}\ell_{y}^{2}(\R\times\T^{d})}^{2})^{\frac{1}{2}}\,.
\end{align}
\end{definition}
\begin{remark}
For all $n\in\Z^{d}\setminus\{0\}$ and $\alpha\ge1$ such that $n\in\mathscr{C}_{\alpha}$, we have
\[
K_{\alpha}\le |n|\le 2K_{\alpha}.
\]
It follows that the Sobolev norms defined in~\eqref{eq:norm-t} are equivalent to the standard Sobolev norms.
\end{remark}
The motivation for introducing the equivalent Sobolev norm $h^{s}(\T^{d})$ is that it does not detect energy transfers within a given cluster. Accordingly, we expect this norm to remain bounded under the effective dynamics~\eqref{eq:eff}.
\begin{definition}[Strong, weak and time-dependent norms]
\label{def:norm}
 Let $s_{0}$ as defined in~\eqref{eq:par},~$s>s_{0}$ and~$\sigma>0$ be arbitrarily small. The strong norm is
\begin{align}
\label{eq:str}
\|F\|_{S}&:= \|(1-\partial_{x}^{2})^\frac{\sigma}{2}F\|_{H_{x,y}^{s}(\R\times\T^{d})}
		+ \|x F\|_{L_{x}^{2}h_{y}^{s}(\R\times\T^{d})} \,,
\intertext{and the weak norm is} 
\label{eq:wea}
\|F\|_{Z}&:=\sup_{\xi\in\R}\|\widehat{F}(\xi)\|_{h_{y}^{s}(\T^{d})}\,.
\intertext{For $\delta>0$ small enough, the space-time norm is }
\label{eq:X}
\|F\|_{X_{T}} &:= \sup_{0\leq t\leq T} 
	\big(
	\|F(t)\|_{Z}
	+
	\langle t\rangle^{-\delta}\|F(t)\|_{S}
	+
	\langle t\rangle^{1-\delta}\|\partial_{t}F(t)\|_{S}
	\big)\,.
\end{align}
\end{definition}
\begin{remark}
\label{rem:n}
Observe that the $Z$-norm is indeed weaker than the strong $S$-norm:
\[
\|F\|_{Z} = \sup_{\xi\in\R}\|\widehat{F}(\xi)\|_{h_{y}^{s}} \lesssim \|\langle x\rangle F\|_{L_{x}^{2}h_{y}^{s}} \lesssim \|F\|_{S}\,.
\]
\end{remark}

\subsection{The Hamiltonian system in the transverse direction}\label{sec:preli}
Let $u\in C^{1}(\R;h^{s}(\T^{d}))$ be a solution on $\T^{d}$ of
\[
i\partial_{t}u + \mathrm{div}(\mathrm{A}\nabla)u = |u|^{2}u\,,
\]
with Fourier coefficients
\[
u_{n}(t)=(2\pi)^{-d}\int_{\T^{d}}\e^{-in\cdot y}u(y)\mathrm{d}y\,,\quad n\in\Z^{d}\,.
\]
The vector $a\in C^{1}(\R;\C^{\Z^{d}})$ defined by 
\[
a_{n}(t):=\e^{it\lambda_{n}^{2}}u_{n}(t)\,,\quad n\in\Z^{d}\,,
\]
is solution to the interaction system
\begin{equation}
\label{eq:sys}
i\partial_{t}a_{n} = \sum_{n_{1}-n_{2}+n_{3}=n}\e^{it\Omega(\vec{n})}a_{n_{1}}\overline{a_{n_{2}}}a_{n_{3}}\,,\quad n\in\Z^{d}\,,
\end{equation}
where we recall that for $\vec{n}=(n_{1},n_{2},n_{3},n)\in(\Z^{d})^{4}$, the resonance function reads
\[
\Omega(\vec{n})=\lambda_{n_{1}}^{2}-\lambda_{n_{2}}^{2}+\lambda_{n_{3}}^{2}-\lambda_{n}^{2}\,.
\]
Under the zero-momentum condition $n_{1}-n_{2}+n_{3}=n$ the resonance function factorizes:
\[
\Omega(\vec{n}) = 2(n_{1}-n_{2})^{\intercal}\mathrm{A}(n_{2}-n_{3})\,.
\]
A subsystem of \eqref{eq:sys} is the resonant system:
\begin{equation}
\label{eq:r-sys}
i\partial_{t}a_{n} = \sum_{(n_{1},n_{2},n_{3})\,:\,\vec{n}\in\Gamma_{0}}a_{n_{1}}\overline{a_{n_{2}}}a_{n_{3}}\,,
\end{equation} 
where we define the four-waves resonant set by
\begin{equation}
\label{eq:4}
\Gamma_{0} =  \big\{
		       \vec{n}\in(\Z^{d})^{4}\quad:\quad n_{1}-n_{2}+n_{3}-n =\Omega(\vec{n})=0 \big\}\,.
\end{equation}
\begin{lemma}[Four-waves resonances]
\label{lem:4w}
Under the Diophantine admissibility condition \eqref{eq:ad} on $\mathrm{A}$, the nonlinear term on the four-waves resonant set is trivial: 
\begin{equation}
\label{eq:4w}
\sum_{\vec{n}
	=(n_{1},n_{2},n_{3},n)\in\Gamma_{0}}a_{n_{1}}\overline{a_{n_{2}}}a_{n_{3}}\overline{a_{n}}
	= 2(\sum_{n\in\Z^{d}}|a_{n}|^{2})^{2} 
	- \sum_{n\in\Z^{d}}|a_{n}|^{4}\,.
\end{equation}
\end{lemma}
\begin{proof}
If we factorize the resonance function as above (using the zero-momentum condition) we obtain from the Diophantine admissibility condition \eqref{eq:ad} on $\mathrm{A}$ that 
\[
(n_{1}-n_{2})^{\intercal}\mathrm{A}(n_{2}-n_{3}) = 0 \quad \implies n_{2}=n_{1}\quad \text{or}\quad n_{2}=n_{3}\,.
\] 
We conclude from the zero-momentum condition that
\[
\Gamma_{0} =\big
	\{
	\vec{n}=(n_{1},n_{2},n_{3},n)\in(\Z^{d})^{4}\ :\ \{n_{1},n_{3}\}=\{n,n_{2}\}
	\big
	\}\,.
	\qedhere
\]
\end{proof}

We now define the effective quasi-resonant interactions. They are those of the form $\mathrm{high\times low\to high}$ where the two high frequencies belong to the same dyadic cluster and the resonance function is smaller than~1. We denote
\begin{equation}
\label{eq:alpha0}
\alpha_{0}:=\min\{\alpha\geq1\ :\ K_{\alpha}\geq10^{10}C(\mathrm{A},d)\}\,,
\end{equation}
where we recall that $K_{\alpha}:=\min_{n\in\mathscr{C}_{\alpha}}|n|$ when $\alpha\geq1$. We refer to high frequency clusters when $\alpha\geq\alpha_{0}$. 
\begin{definition}[Effective quasi-resonant interactions]
\label{def:Lamb}
For $\alpha\geq\alpha_{0}$ as above and for all $n\in\mathscr{C}_{\alpha}$, we let
\begin{multline*}
\Lambda^{(1)}(n):=
	\big\{
	(n_{1},n_{2},n_{3})
	\in(\Z^{d})^{3}
	\ :\  n_{1}-n_{2}+n_{3}=n\,,\quad
	 |\Omega(\vec{n})|<1\,,\\ 
	 n_{1}\in\mathscr{C}_{\alpha}\,,\quad 
	 |n_{2}|,|n_{3}| < 10^{-5}K_{\alpha}
	\big\}
\end{multline*}
and 
\begin{multline*}
\Lambda^{(3)}(n):=
	\big\{
	(n_{1},n_{2},n_{3})
	\in(\Z^{d})^{3}
	\ :\  n_{1}-n_{2}+n_{3}=n\,,\quad
	 |\Omega(\vec{n})|<1\,,\\ 
	 n_{3}\in\mathscr{C}_{\alpha}\,,\quad 
	 |n_{2}|,|n_{1}| < 10^{-5}K_{\alpha}
	\big\} \,.
\end{multline*}
The effective quasi-resonant interaction equation is as follows:  
\begin{itemize}
\item If $n\in\bigcup_{\alpha\geq\alpha_{0}}\mathscr{C}_{\alpha}$ then 
\[
i\partial_{t}a_{n}
	= 
	2\sum_{(n_{1},n_{2},n_{3})\in\Lambda^{(1)}(n)}
	\e^{it\Omega(\vec{n})}a_{n_{1}}\overline{a_{n_{2}}}a_{n_{3}}\,.
\]
\item Otherwise, 
\[
i\partial_{t}a_{n} = 0\,.
\]
\end{itemize}
\end{definition}
\begin{remark}In the definition of the effective quasi-resonant interaction we implicitly used that $n_{1}$ and $n_{3}$ play symmetric roles, hence the factor 2 (instead of summing over $\Lambda^{(1)}(n)$ and over $\Lambda^{(3)}(n)$). For the $\mathrm{high\times low \to high}$ quasi-resonant interactions with an outgoing wave in $\bigcup_{\alpha<\alpha_{0}}\mathscr{C}_{\alpha}$ the small divisors lose a harmless constant. They do not contribute to the effective dynamics. 
\end{remark}
We now prove that the effective system preserves the high-frequency super-actions based on the dyadic clusters. Recall that high-frequencies refer to the clusters $\mathscr{C}_{\alpha}$ with $\alpha\geq\alpha_{0}$ defined in \eqref{eq:alpha0}. By definition, if $(n_{1},n_{2},n_{3})\in\Lambda^{(1)}(n)$ then $n_{2},n_{3}\notin\mathscr{C}_{\alpha}$.
\begin{lemma}\label{lem:cons}
Suppose that $a\in C(\R;\ell^{2}(\T^{d}))$ is decomposed into
\[
a(t,y) = \sum_{n\in\Z^{d}}a_{n}(t)\e^{in\cdot y}\,,\quad y\in\T^{d}\,,
\]
where $(a_{n}(t))_{n\in\Z^{d}}$ is solution to the system in Definition \eqref{def:Lamb}. Then, for all $\alpha\geq\alpha_{0}$,  we have
\[
\frac{\mathrm{d}}{\mathrm{d}t}\|\pi_{\alpha}a(t)\|_{\ell^{2}(\T^{d})}^{2} = 0\,.
\]
\end{lemma}
\begin{proof}
Let $\alpha\geq\alpha_{0}$. We have 
\begin{align*}
\frac{1}{4}\frac{\mathrm{d}}{\mathrm{d}t}\|\pi_{\alpha}a(t)\|_{\ell^{2}(\T^{d})}^{2}
	&= 
	\im
	 \big(
	\sum_{n\in\mathscr{C}_{\alpha}}\overline{a_{n}}
	\sum_{(n_{1},n_{2},n_{3})\in\Lambda^{(1)}(n)}
	\mathbf{1}_{|\Omega(\vec{n})|<1}\e^{it\Omega(\vec{n})}
	a_{n_{1}}\overline{a_{n_{2}}}a_{n_{3}} \big)\\
	&=\im
	 \big(
	\sum_{n,n_{1}\in\mathscr{C}_{\alpha}}\sum_{
	|n_{2}|,|n_{3}|\leq 10^{-5}K_{\alpha}}
	\mathbf{1}_{\substack{|\Omega(\vec{n})|<1\\n_{1}+n_{3}=n+n_{2}}}
	\e^{it\Omega(\vec{n})}\overline{a_{n}}a_{n_{1}}\overline{a_{n_{2}}}a_{n_{3}}
	 \big)\\
	&=0\,.
\end{align*}
To see this cancellation, one can take the complex conjugate and interchange the roles of $n$ with $n_{1}$ (resp. $n_{2}$ with $n_{3}$).
\end{proof}
In the following lemma, we use the frequency separation property of Lemma~\ref{lem:sep} to control the small-divisor losses that appear in the Poincaré--Dulac normal form reduction of Section~\ref{sec:fast}.
We defined the set $\Lambda^{(1)}(n)$ and $\Lambda^{(3)}(n)$ in Definition \ref{def:Lamb}, and the high-frequency threshold~$\alpha_{0}$ in~\eqref{eq:alpha0}.
\begin{lemma}\label{lem:small-div}  Let $s\geq0$. There exists $C>0$ (depending on $d$, $s$ and $\mathrm{A}$) such that for all $\alpha\geq\alpha_{0}$, $n\in\mathscr{C}_{\alpha}$ and 
\[
(n_{1},n_{2},n_{3})\in(\Z^{d})^{3}\setminus(\Lambda^{(1)}(n)\cup\Lambda^{(3)}(n)),\quad \text{with}\quad n_{1}-n_{2}+n_{3}=n\,,\quad \Omega(\vec{n})\neq0\,,
\]
we have 
\[
\frac{1}{|\Omega(\vec{n})|}K_{\alpha}^{s}\leq C(n_{1}^{\ast})^{s}(n_{2}^{\ast})^{\frac{4\tau}{c(d)}}\,,
\]
where $n_{1}^{\ast}\geq n_{2}^{\ast}\geq n_{3}^{\ast}$ is the decreasing rearrangement of $|n_{1}|,|n_{2}|,|n_{3}|$. 
\end{lemma}
\begin{proof} 
Let $\alpha\geq\alpha_{0}$ and $\vec{n}=(n_{1},n_{2},n_{3},n)$, with $n\in\mathscr{C}_{\alpha}$ and $n_{1}-n_{2}+n_{3}=n$ (zero-momentum condition). In particular, 
\[
K_{\alpha}\leq |n|\leq 2K_{\alpha}\,,\quad \text{and}\quad K_{\alpha}\leq |n|\leq 3n_{1}^{\ast}\,.
\]
When $|\Omega(\vec{n})|\geq1$, the desired bound follows from the above inequalities between the indices.  Then, according to the Diophantine condition \eqref{eq:ad} we can suppose that 
\[
1<\frac{1}{|\Omega(\vec{n})|} \lesssim (n_{1}^{\ast})^{2\tau}\,.
\]
Moreover, we suppose without loss of generality that $|n_{1}|\geq |n_{3}|$  since they play symmetric roles. Hence, $|n_{1}|\geq n_{2}^{\ast}$. We separate three cases. 
\medskip

\noindent{$\bullet$ \bf Case 1:} $|n_{1}|=n_{2}^{\ast}$. Suppose first that $n_{1}^{\ast}=|n_{2}|$. The condition $|\Omega(\vec{n})|<1$ together with the coercive bound on the eigenvalues \eqref{eq:co} yield
\[
(n_{1}^{\ast})^{2}\lesssim_{\mathrm{A}} \lambda_{n_{2}}^{2}+\lambda_{n}^{2} \leq |\Omega(\vec{n})| + \lambda_{n_{1}}^{2}+\lambda_{n_{3}}^{2} \lesssim_{\mathrm{A}} (n_{2}^{\ast})^{2}\,,
\]

for some implicit constants only depending on $\mathrm{A}$. 

Otherwise, since $|n_{1}|\geq |n_{3}|$ we have $|n_{1}|=|n_{3}|=n_{1}^{\ast}=n_{2}^{\ast}$. In both cases, $n_{1}^{\ast}\lesssim_{\mathrm{A}} n_{2}^{\ast}$ and we obtain 
\[
\frac{1}{|\Omega(\vec{n})|}K_{\alpha}^{s} \leq (n_{1}^{\ast})^{s}(n_{1}^{\ast})^{2\tau} \lesssim_{\mathrm{A}} (n_{1}^{\ast})^{s}(n_{2}^{\ast})^{2\tau}\,\]
which is conclusive (we recall that $c(d)\leq2$, so that $2\tau\leq\frac{4\tau}{c(d)}$).  
\medskip

\noindent{$\bullet$ \bf Case 2:} $|n_{1}|=n_{1}^{\ast}$ and $n_{1}\notin \mathscr{C}_{\alpha}$. The frequency separation property \eqref{eq:sep} gives 
 \[
 |\lambda_{n}^{2}-\lambda_{n_{1}}^{2}| + |n-n_{1}| \geq  (n_{1}^{\ast})^{c(d)}\,.
 \]
We deduce from the condition $|\Omega(\vec{n})|<1$ and the zero-momentum condition that
 \[
 |\lambda_{n_{2}}^{2}-\lambda_{n_{3}}^{2}| + |n_{2}-n_{3}|
 	\geq 
	|\lambda_{n}^{2}-\lambda_{n_{1}}^{2}| + |n-n_{1}|
	- |\Omega(\vec{n})| 
	\geq
	(n_{1}^{\ast})^{c(d)} - 1 \geq \frac{1}{2}(n_{1}^{\ast})^{c(d)}\,, 
 \]
 which in turn implies that 
 \[
n_{2}^{\ast} = |n_{2}|\vee|n_{3}| \gtrsim_{\mathrm{A}} (n_{1}^{\ast})^{\frac{c(d)}{2}}\,.
 \]
 Hence,
 \[
 \frac{1}{|\Omega(\vec{n})|}K_{\alpha}^{s} 	\lesssim_{\mathrm{A}}
 	(n_{1}^{\ast})^{s+2\tau}
 \lesssim_{\mathrm{A}} (n_{1}^{\ast})^{s}(n_{2}^{\ast})^{\frac{4\tau}{c(d)}}\,,
 \]
 which is conclusive. 
\medskip

\noindent{$\bullet$ \bf Case 3:} $|n_{1}|=n_{1}^{\ast}$ and $n_{1}\in \mathscr{C}_{\alpha}$. In particular, $n_{1}^{\ast}\geq |n_{1}|\geq K_{\alpha}$. Moreover, since~$(n_{1},n_{2},n_{3})~\notin~\Lambda^{(1)}(n)$ we have from the Definition \ref{def:Lamb} that  
\[
10^{-5}K_{\alpha} \leq \max(|n_{2}|,|n_{3}|)\,.
\]
Therefore, $K_{\alpha}\lesssim n_{2}^{\ast}$. Since $n_{1}\in\mathscr{C}_{\alpha}$ forces $n_{1}^{\ast}=|n_{1}|\leq2K_{\alpha}$, we have $n_{1}^{\ast}\sim n_{2}^{\ast}\sim K_{\alpha}$. We conclude that
\[
\frac{1}{|\Omega(\vec{n})|}K_{\alpha}^{s} \lesssim (n_{1}^{\ast})^{2\tau}(n_{2}^{\ast})^{s} \lesssim (n_{1}^{\ast})^{s}(n_{2}^{\ast})^{\frac{4\tau}{c(d)}}\,.
\]
This completes the proof of Lemma \ref{lem:small-div}. 
\end{proof}

\subsection{Dispersion in the unbounded direction} Now that we have analyzed some interactions of the system in the compact direction, we turn to the dispersive properties of the Schrödinger free evolution in the unbounded direction $\R$. In the following lemma we recall a dispersive estimate and its consequence in our context. 

\begin{lemma}\label{lem:disp} There exists $C>0$ such that for all $f\in L^{2}(\R)\cap x^{-1}L^{2}(\R)$ and~$t>0$,
\begin{equation}
\label{eq:disp}
	\left\|
	\e^{it\partial_{x}^{2}}f(x) 
	- \left(\frac{\pi}{it}\right)^{\frac{1}{2}}\e^{i\frac{x^{2}}{4t}}\widehat{f}(\frac{x}{2t})
	\right\|_{L_{x}^{\infty}(\R)}
	\leq C 
	 t^{-\frac{3}{4}}\|xf\|_{L_{x}^{2}(\R)}\,.
\end{equation}
Considering the profile $F(t)=\e^{-it\Delta}U(t)$, we have that for all $T\geq1$ and $t\in[1,T]$, 
\begin{equation}
\label{eq:key}
t^{\frac{1}{2}}\|U(t)\|_{L_{x}^{\infty}h_{y}^{s}}
	\lesssim 
	\|F\|_{Z} 
	+
	t^{-\frac{1}{4}}\|F\|_{S} \lesssim \|F\|_{X_{T}}\,.
\end{equation}
\end{lemma}
\begin{remark} The above estimate~\eqref{eq:key} is a key ingredient in our analysis, especially to obtain the trilinear estimate claimed in Lemma \ref{lem:sb}. Such an estimate is not available in the case of the square torus. As proved in \cite[Theorem 1.1]{HPTV15}, in this case the space-time norm $X_{T}$ provides the dispersive estimate for the $L_{x}^{\infty}h_{y}^{1}(\R\times\T^{d})$-norm only.
\end{remark}

\begin{proof}
The dispersive estimate \eqref{eq:disp} on $\R$ is standard (see for instance Hayashi--Naumkin \cite{HN98}), and follows from the expression of the semi-group $\e^{it\partial_{x}^{2}}$. We write a short proof in Appendix \ref{app}.
We now prove \eqref{eq:key}: for all $(x,y)\in\R\times\T^{d}$, we expand
\[
U(t,x,y) = \sum_{n\in\Z^{d}}\e^{iy\cdot n-it\lambda_{n}^{2}}\e^{it\partial_{x}^{2}}F_{n}(x)\,.
\]
Using the equivalence of the $h^{s}$-norms with the standard Sobolev norms on $\T^{d}$ we have
\[
\|U(t)\|_{L^{\infty}h^{s}(\R\times\T^{d})} \lesssim_{s} 
	\big\|
	\big(\sum_{n\in\Z^{d}} \langle n\rangle^{2s} |\e^{it\partial_{x}^{2}}F_{n}(x)|^{2}
	\big)^{\frac{1}{2}}
	\big\|_{L^{\infty}(\R)}\,.
\]
According to the dispersive estimate \eqref{eq:disp}, we have that  for fixed $x\in\R$, $n\in\Z^{d}$ and $t>0$,
\[
|\e^{it\partial_{x}^{2}}F_{n}(x)|^{2} \lesssim t^{-1}|\widehat{F}_{n}(\frac{x}{2t})|^{2} + t^{-\frac{3}{2}}\|xF_{n}\|_{L_{x}^{2}(\R)}^{2}\,.
\]
When summing over $n$, the first contribution is bounded by 
\[
t^{-1}\sum_{n\in\Z^{d}}\langle n\rangle^{2s}|\widehat{F}_{n}(\frac{x}{2t})|^{2} 
	=
	t^{-1}
	\|\widehat{F}(\frac{x}{2t})\|_{h^{s}(\T^{d})}^{2}
	\leq 
	t^{-1}\sup_{\xi\in\R}
	\| \widehat{F}(\xi)\|_{h^{s}(\T^{d})}^{2}\,.
\]
For the second contribution, we have 
\begin{align*}
\sum_{n}\langle n\rangle^{2s}\|xF_{n}\|_{L_{x}^{2}(\R)}^{2} 
	&\lesssim_{s}
	\|xF\|_{L_{x}^{2}h_{y}^{s}(\R\times\T^{d})}^{2}\,.
\end{align*}
This proves that for all $t>0$, 
\begin{equation}
 \label{eq:disp1}
\|U(t)\|_{L^{\infty}h^{s}(\R\times\T^{d})} \lesssim t^{-\frac{1}{2}}\sup_{\xi\in\R}\ \|\widehat{F}(\xi)\|_{h^{s}(\T^{d})}
	+ t^{-\frac{1}{2}-\frac{1}{4}}
	\|x F\|_{L_{x}^{2}h_{y}^{s}(\R\times\T^{d})}\,,
\end{equation}
which gives \eqref{eq:key} provided $\delta<\frac{1}{4}$ in the definition of the $X_{T}$-norm.
\end{proof}

\subsection{Profile equation}

We detail the standard preparation in which we pull-back the nonlinear solution $U$ by the linear flow:
\begin{align*}
U(t,x,y) 
	&=\e^{it\Delta_{\mathrm{A}}}F(t,x,y)\,.
\end{align*}
The profile $F(t)$ is solution to the interaction equation 
\[
i\partial_{t}F(t) 
	= \e^{-it\Delta_{\mathrm{A}}}(\e^{it\Delta_{\mathrm{A}}}F(t)\e^{-it\Delta_{\mathrm{A}}}\overline{F(t)}\e^{it\Delta_{\mathrm{A}}}F(t))\,.
\]
We write
\begin{equation}
\label{eq:N}
	i\partial_{t}F(t) = \mathcal{N}^{t}[F(t),F(t),F(t)]\,,
\end{equation}
where the trilinear form $\mathcal{N}^{t}$ on $H^{s}(\R\times\T^{d})$ is
\[
\mathcal{N}^{t}[F,G,H] := \e^{-it\Delta_{\mathrm{A}}}(\e^{it\Delta_{\mathrm{A}}}F\e^{-it\Delta_{\mathrm{A}}}\overline{G}\e^{it\Delta_{\mathrm{A}}}H)\,.
\]
When the context is clear we denote 
\[
\mathcal{N}^{t}[F]:=\mathcal{N}^{t}[F(t),F(t),F(t)]\,.
\]
Recalling the notation \eqref{eq:f-f}, the full spatial Fourier transform  of the trilinear interaction at $(\xi,n)\in\R\times\Z^{d}$ reads
\begin{multline*}
\mathcal{F}\mathcal{N}^{t}[F,G,H](\xi,n) 
	:= \sum_{n_1-n_2+n_3=n}
	\int_{\R^{2}}
	\e^{it(\Omega(\vec{n})
	+2(\eta-\xi)(\zeta-\eta))}\\
	\times
	\widehat{F}_{n_{1}}(\xi-\eta+\zeta)\overline{\widehat{G}_{n_{2}}}(\zeta)\widehat{H}_{n_{3}}(\eta)
	\mathrm{d}\eta \mathrm{d}\zeta\,,
\end{multline*}
where the sums run over $(n_{1},n_{2},n_{3})\in(\Z^{d})^{3}$. Once again we denoted 
\[
\vec{n}=(n_{1},n_{2},n_{3},n)\,,\quad \Omega(\vec{n}) := \lambda_{n_{1}}^{2} -\lambda_{n_{2}}^{2}+\lambda_{n_{3}}^{2}-\lambda_{n}^{2}
\,.
\]
Changing variables in the integrals over $(\eta,\zeta)\in\R^{2}$ we obtain 
\begin{multline}
\label{eq:ast}
\mathcal{F}\mathcal{N}^{t}[F,G,H](\xi,n) 
	= \sum_{n_1-n_2+n_3=n}
	\int_{\R^{2}}
	\e^{it(\Omega(\vec{n})+
	2\eta\zeta)}
	\\
	\times\widehat{F}_{n_{1}}(\xi-\zeta)\overline{\widehat{G}_{n_{2}}}(\xi-\eta-\zeta)\widehat{H}_{n_{3}}(\xi-\eta)
	\mathrm{d}\eta \mathrm{d}\zeta\,. 
\end{multline}
\subsection{A priori estimates for the strong norm}  We conclude this preparatory section with the  proof of an a priori trilinear bound for the $S$-norm. 
\begin{lemma}\label{lem:sb}
Suppose that the profile $F$ is solution to \eqref{eq:N}. Then, for all $T\geq0$ and $t\in[0,T]$,
\begin{equation}
\label{eq:sb3}
\|\mathcal{N}^{t}[F(t),F(t),F(t)]\|_{S}\lesssim\langle t\rangle^{-1+\delta}\|F\|_{X_{T}}^{3}\,.
\end{equation}
\end{lemma}

This bound is a stronger analogue of \cite[Lemma 2.1]{HPTV15}. Indeed, we can readily use the dispersive estimate \eqref{eq:key}, which does not hold in the case of the square torus, to obtain a control of the $L^{\infty}(\R\times\T^{d})$-norm by the $X_{T}$-norm with the optimal decay in $t^{-\frac{1}{2}}$. 

\begin{proof}
Recall that  
\[
\mathcal{N}^{t}[F]:=\mathcal{N}^{t}[F(t),F(t),F(t)] = \e^{-it\Delta_{\mathrm{A}}}(|U(t)|^{2}U(t))\,.
\]
The proof when $t\in[0,1]$ is easier since we do not need to gain decay in time, and we only detail the case when $T\geq1$ and $1\leq t\leq T$. We have from the product rule of Lemma \ref{lem:prod} and the dispersive bound \eqref{eq:key} that for $\sigma\geq0$,
\begin{multline*}
\|(1-\partial_{x}^{2})^{\frac{\sigma}{2}}\mathcal{N}^{t}[F]\|_{H_{x,y}^{s}}
	\lesssim \|U(t)\|_{L_{x}^{\infty}h_{y}^{s}}^{2}\|(1-\partial_{x}^{2})^{\frac{\sigma}{2}}U(t)\|_{H_{x,y}^{s}}
	\\
	\lesssim t^{-1}\|F\|_{X_{T}}^{2}\|(1-\partial_{x}^{2})^{\frac{\sigma}{2}}F(t)\|_{H_{x,y}^{s}}\,. 
\end{multline*}
We deduce from the definition \eqref{eq:X} of the $X_{T}$-norm that 
\[
\|(1-\partial_{x}^{2})^{\frac{\sigma}{2}}\mathcal{N}^{t}[F]\|_{H_{x,y}^{s}}\lesssim t^{-1+\delta}\|F\|_{X_{T}}^{3}\,.
\]
As for the second component of the $S$-norm we have 
\[
\|x\mathcal{N}^{t}[F]\|_{L_{x}^{2}h_{y}^{s}} 
	=
	\|\partial_{\xi}\mathcal{F}\mathcal{N}^{t}[F]\|_{L_{x}^{2}h_{y}^{s}}\,.
\]
For fixed $n\in\Z^{d}$, following \cite{KP11}, we have from the expression \eqref{eq:ast} that 
\begin{multline*}
\partial_{\xi}\mathcal{F}\mathcal{N}^{t}[F](\xi,n) 
	= \sum_{n_{1}-n_{2}+n_{3}=n}\int_{\R^{2}}
	\e^{it(\Omega(\vec{n})+2\eta\zeta)}
	\\
	\partial_{\xi}\left[ 
	\widehat{F}_{n_{1}}(t,\xi-\zeta)\overline{\widehat{F}_{n_{2}}}(t,\xi-\eta-\zeta)\widehat{F}_{n_{3}}(t,\xi-\eta)
	\right]
	\mathrm{d}\eta \mathrm{d}\zeta\,.
\end{multline*}	
Applying the Leibnitz rule and redistributing the phases, we obtain that
\[
\|x\mathcal{N}^{t}[F(t)]\|_{L_{x}^{2}h_{y}^{s}}
\lesssim
\|\e^{-it\Delta_{\mathrm{A}}}(|U(t)|^{2}\e^{it\Delta_{\mathrm{A}}}(xF(t)))\|_{L_{x}^{2}h_{y}^{s}}\,.
\]
We deduce from the dispersive bound \eqref{eq:key} that
\[
\|x\mathcal{N}^{t}[F(t)]\|_{L_{x}^{2}h_{y}^{s}} 
\lesssim 
\|U(t)\|_{L_{x}^{\infty}h_{y}^{s}}^{2}
\|xF(t)\|_{L_{x}^{2}h_{y}^{s}}
\lesssim 
t^{-1+\delta}\|F\|_{X_{T}}^{3}\,,
\]
which finishes the proof of Lemma \ref{lem:sb}. 
\end{proof}

\section{Nonlinear interactions}\label{sec:non}
Lemma~\ref{lem:sb} provides an a priori estimate of the strong $S$-norm  in terms of the~$X_T$-norm of the solution, with the appropriate time decay. To close the argument, it remains to show that the  $Z$-norm is integrable in time.
\medskip

To this end, we will analyze the trilinear interaction $\mathcal{N}^{t}$ more carefully.  We first isolate the space-resonant interactions (denoted $\mathcal{N}_{\R,\mathrm{res}}^{t}$) and we use space oscillations to prove that the contribution of the rest (denoted $\mathcal{N}_{\R,\mathrm{nr}}^{t}$) is acceptable. 
\medskip

Then, we further decompose $\mathcal{N}_{\R,\mathrm{res}}^{t}$ into an effective Hamiltonian system (denoted $\mathcal{N}_{\mathrm{eff}}^{t}$) in the transverse direction $\T^{d}$ and a remainder interaction (denoted~$\mathcal{E}^{t}$) oscillating sufficiently fast in time. 
\subsection{Decomposition of nonlinear interactions} Following \cite{KP11}, we apply the Plancherel's identity in \eqref{eq:ast} and use the explicit expression of the Fourier transform of a Gaussian function to isolate the space-resonant component of the nonlinearity. We obtain (see also the proof of \cite[Lemma 3.10]{HPTV15}) that for fixed $\xi\in\R$, $t\geq1$ and $(n_{1},n_{2},n_{3})\in(\Z^{d})^{3}$,
\begin{align*}
\int_{\R^{2}}
	&\e^{i2t\eta\zeta}
	\widehat{F}_{n_{1}}(\xi-\zeta)
	\overline{\widehat{G}_{n_{2}}}(\xi-\eta-\zeta)
	\widehat{H}_{n_{3}}(\xi-\eta)
	\mathrm{d}\eta \mathrm{d}\zeta \\
	&=\frac{1}{(2\pi)^{3}}\frac{\pi}{t}
	\int_{\R^{3}} F_{n_{1}}(x_{1})\overline{G_{n_{2}}}(x_{2})H_{n_{3}}(x_{3})
	\e^{-i\xi (x_{1}-x_{2}+x_{3})}\e^{
	i[
	\frac{1}{2t}(x_{1}-x_{2})(x_{2}-x_{3})
	]}
	\mathrm{d}x_{1}\mathrm{d}x_{2}\mathrm{d}x_{3} \\
	&=
	\frac{\pi}{t}\widehat{F}_{n_{1}}(\xi)\overline{\widehat{G}_{n_{2}}}(\xi)\widehat{H}_{n_{3}}(\xi)
	+
	\mathcal{O}^{t}[F_{n_{1}},G_{n_{2}},H_{n_{3}}](\xi)
	\,,
\end{align*}  
where $\mathcal{O}^{t}$ collects the spatially non-resonant interactions: for functions $f,g,h$ in~$H_{x}^{s}(\R)$,
\begin{multline}
\label{eq:O}
\mathcal{O}^{t}[f,g,h](\xi) :=\frac{1}{(2\pi)^{3}}\frac{\pi}{t}
	\int_{\R^{3}} f(x_{1})\overline{g(x_{2})}h(x_{3})
	\e^{-i\xi (x_{1}-x_{2}+x_{3})}\\
	 \big(
	\e^{
	i[
	\frac{1}{2t}(x_{1}-x_{2})(x_{2}-x_{3})
	]}-1
	 \big)
	\mathrm{d}x_{1}\mathrm{d}x_{2}\mathrm{d}x_{3}\,.
\end{multline}
One can prove that $\mathcal{O}^{t}$ has an improved decay in time under some assumptions on the space localization of the functions $f,g,h$. We detail this part of the analysis in Section~\ref{sec:snr}. 
\medskip

Recalling the expression \eqref{eq:ast}, we have that
for $t\geq1$ and $(\xi,n)\in\R\times\Z^{d}$,
\begin{multline}
\label{eq:d-fo}
\mathcal{F}\mathcal{N}^{t}[F,G,H](\xi,n)
	= \frac{\pi}{t}\sum_{n_{1}-n_{2}+n_{3}=n}\e^{it\Omega(\vec{n})}
	\widehat{F}_{n_{1}}(\xi)\overline{\widehat{G}_{n_{2}}}(\xi)\widehat{H}_{n_{3}}(\xi)  \\
	+
	\sum_{n_{1}-n_{2}+n_{3}=n}\e^{it\Omega(\vec{n})}
	\mathcal{O}^{t}[F_{n_{1}},G_{n_{2}},H_{n_{3}}](\xi)\,.
\end{multline}
According to this decomposition we write
\begin{equation}
\label{eq:nl}
\mathcal{F}\mathcal{N}^{t}[F,G,H](\xi,n) 
	=: \mathcal{F}\mathcal{N}_{\R,\mathrm{res}}^{t}[F,G,H](\xi,n)
	+\mathcal{F}\mathcal{N}_{\R,\mathrm{nr}}^{t}[F,G,H](\xi,n) \,.
\end{equation}
We analyze in section~\ref{sec:sr} the term $\mathcal{N}^{t}_{\R,\mathrm{res}}$, which regroups the space-resonant interactions, and the term $\mathcal{N}^{t}_{\R,\mathrm{nr}}$ in section~\ref{sec:snr}, respectively. First, we further decompose the space-resonant system  involving $\mathcal{N}^{t}_{\R,\mathrm{res}}$ into two parts: an effective system consisting of resonant interactions ($\mathcal{R}^{t}$), some quasi-resonant interactions adapted to the frequency cluster of the outgoing wave ($\mathcal{Q}^{t}$), and a remainder ($\mathcal{E}^{t}$):
\[
 \mathcal{F}\mathcal{N}_{\R,\mathrm{res}}^{t}[F,G,H]
	:= 
	\mathcal{F}\mathcal{R}^{t}[F,G,H]
	+
	\mathcal{F}\mathcal{Q}^{t}[F,G,H]
	+ \mathcal{F}\mathcal{E}^{t}[F,G,H]\,.
\]
The effective system is
\begin{equation}
\label{eq:eff}
\mathcal{N}_{\mathrm{eff}}^{t} = \mathcal{R}^{t} + \mathcal{Q}^{t}\,. 
\end{equation}
Recalling Definition \ref{def:Lamb} of $\Lambda^{(1)}(n)$ and $\Lambda^{(3)}(n)$, the different interactions are defined as follows:
\begin{itemize}
\item {\it Resonant interaction:} since $\mathcal{R}^{t}$ is defined here as a multilinear form, $F$ and~$H$ no longer play symmetric roles and the two resonant branches $\{n_{1},n_{3}\}=\{n,n_{2}\}$ must be kept separately,
\begin{multline}
\label{eq:R0}
\mathcal{F}\mathcal{R}^{t}[F,G,H](\xi,n)
	:=\frac{\pi}{t}
	\sum_{n_{1}\in\Z^{d}}\widehat{F}_{n_{1}}(\xi)\overline{\widehat{G}_{n_{1}}}(\xi)\widehat{H}_{n}(\xi)
	+\frac{\pi}{t}
	\sum_{n_{3}\in\Z^{d}}\widehat{F}_{n}(\xi)\overline{\widehat{G}_{n_{3}}}(\xi)\widehat{H}_{n_{3}}(\xi)\\
	-\frac{\pi}{t}\widehat{F}_{n}(\xi)\overline{\widehat{G}_{n}}(\xi)\widehat{H}_{n}(\xi)\,.
\end{multline}
\item {\it Effective quasi-resonant interactions:} likewise, we sum over both branches $\Lambda^{(1)}(n)$ and $\Lambda^{(3)}(n)$: for $n\in \bigcup_{\alpha\geq\alpha_{0}}\mathscr{C}_{\alpha}$,
\begin{equation}
\label{eq:Q}
\mathcal{F}\mathcal{Q}^{t}[F,G,H](\xi,n)
	=
	\frac{\pi}{t}\sum_{(n_{1},n_{2},n_{3})\in\Lambda^{(1)}(n)\cup\Lambda^{(3)}(n)}\e^{it\Omega(\vec{n})}
	\widehat{F}_{n_{1}}(\xi)\overline{\widehat{G}_{n_{2}}}(\xi)\widehat{H}_{n_{3}}(\xi)\,.
\end{equation}
\item When $n\in\bigcup_{\alpha<\alpha_{0}}\mathscr{C}_{\alpha}$ is a low frequency, we let
\[
\mathcal{F}\mathcal{Q}^{t}[F,G,H](\xi,n) 
	=0\,.
\]
\item {\it Remainder:}
\begin{equation}
\label{eq:Eo}
\mathcal{F}\mathcal{E}^{t}[F,G,H]:= \mathcal{F}\mathcal{N}_{\R,\mathrm{res}}^{t}[F,G,H] - \mathcal{F}\mathcal{R}^{t}[F,G,H] - 	\mathcal{F}\mathcal{Q}^{t}[F,G,H]\,.
\end{equation}
\end{itemize}
\begin{remark}
In particular, when $n\in\bigcup_{\alpha\geq\alpha_{0}}\mathscr{C}_{\alpha}$, Lemma~\ref{lem:4w} shows that removing the exact resonant branches is equivalent to imposing $n_{1},n_{3}\neq n$. Hence
\[
\mathcal{F}\mathcal{E}^{t}[F,G,H](\xi,n) 
	:= \frac{\pi}{t}\sum_{\substack{(n_{1},n_{2},n_{3})\notin\Lambda^{(1)}(n)\cup\Lambda^{(3)}(n)\\n_{1}-n_{2}+n_{3}=n\\
	n_{1},n_{3}\neq n}}\e^{it\Omega(\vec{n})}
	\widehat{F}_{n_{1}}(\xi)\overline{\widehat{G}_{n_{2}}}(\xi)\widehat{H}_{n_{3}}(\xi)\,. 
\]
\end{remark}
\subsection{The space resonant interactions}
\label{sec:sr}
We analyze the system involving~$\mathcal{N}^{t}_{\R,\mathrm{res}}$, decomposed into two parts. The effective system defined in \eqref{eq:s-eff} preserves the super-actions based on the frequency cluster decomposition, while the remainder term~$\mathcal{E}^{t}$ defined in~\eqref{eq:Eo} oscillates sufficiently fast in time to produce an integrable decay.

\subsubsection{Trilinear bounds}
Before analyzing the individual contributions of $\mathcal{R}^{t}$, $\mathcal{Q}^{t}$, and~$\mathcal{E}^{t}$, we prove trilinear bounds for their sum $\mathcal{N}^{t}_{\R,\mathrm{res}}$.

\begin{lemma}
\label{lem:tri-res} For all $t>0$,
\begin{align}
\notag
	\|
	\mathcal{N}_{\R,\mathrm{res}}^{t}[F^{a},F^{b},F^{c}]
	\|_{Z}
	&\lesssim t^{-1}\|F^{a}\|_{Z}\|F^{b}\|_{Z}\|F^{c}\|_{Z}\,,
	\\
	\notag
	\|
	\mathcal{N}_{\R,\mathrm{res}}^{t}[F^{a},F^{b},F^{c}]
	\|_{S}
	&\lesssim t^{-1} \sum_{\sigma\in\mathfrak{S}_{3}} 
	\|F^{\sigma(a)}\|_{Z}\|F^{\sigma(b)}\|_{Z}\|F^{\sigma(c)}\|_{S}\,,\\
		\label{eq:boS}
	\|
	\mathcal{N}_{\R,\mathrm{res}}^{t}[F^{a},F^{b},F^{c}]
	\|_{H_{x,y}^{s}}
	&\lesssim t^{-1} \sum_{\sigma\in\mathfrak{S}_{3}} 
	\|F^{\sigma(a)}\|_{Z}\|F^{\sigma(b)}\|_{Z}\|F^{\sigma(c)}\|_{H_{x,y}^{s}}\,,
\end{align}
where $\mathfrak{S}_{3}$ is the group of permutations of three elements. The same estimates hold for the subsystems  involving $\mathcal{R}^{t}$, $\mathcal{Q}^{t}$ and $\mathcal{E}^{t}$. 
\end{lemma}

\begin{proof} For fixed $\xi\in\R$ and $t>0$ we have from the Young's convolution inequality that
\[
\|\mathcal{F}\mathcal{N}_{\R,\mathrm{res}}^{t}[F^{a},F^{b},F^{c}](\xi)\|_{h_{y}^{s}}
	\lesssim t^{-1}
	\sum_{\sigma\in\mathfrak{S}_{3}}\|\widehat{F^{\sigma(a)}}(\xi)\|_{\ell_{y}^{1}}
	\|\widehat{F^{\sigma(b)}}(\xi)\|_{\ell_{y}^{1}}
	\|\widehat{F^{\sigma(c)}}(\xi)\|_{h_{y}^{s}}\,,
\]
and the bounds follow from the assumption $s>\frac{d}{2}$. 
\end{proof}

\subsubsection{The effective system}
\label{sec:eff}
We now handle the contribution of the effective interactions defined in \eqref{eq:eff}. Here, we use the key property that the effective system preserves the high-frequency super-actions, as proved in Lemma \ref{lem:cons}. 

\begin{lemma}\label{lem:eff}
Suppose that $G\in C([1,T];H^{s}(\R\times\T^{d}))$ is a solution to the effective system and that $\sup_{t\in[1,T]}\|G(t)\|_{Z}<\infty$.
\begin{equation}
\label{eq:s-eff}
i\partial_{t}G(t) = \mathcal{N}_{\mathrm{eff}}^{t}[G(t)]\,.
\end{equation} 
Then, for all $t\in[1,T]$, $\alpha\in\N$ and $\xi\in\R$ we have 
\begin{equation}
\label{eq:cs}
	\|\pi_{\alpha}\widehat{G}(t,\xi)\|_{\ell^{2}(\T^{d})} 
	= \|\pi_{\alpha}\widehat{G}(1,\xi)\|_{\ell^{2}(\T^{d})}\,.
\end{equation}
Moreover, 
\begin{equation}
\label{eq:cs1}
\|G(t)\|_{Z} = \|G(1)\|_{Z}\,,\quad \|G(t)\|_{H^{s}(\R\times\T^{d})} = \|G(1)\|_{H^{s}(\R\times\T^{d})}\,.
\end{equation}
\end{lemma}
\begin{proof} 
Fix $\xi\in\R$, and consider 
\[
a(\xi;t,y) := \widehat{G}(t,\xi,y) = \sum_{n\in\Z^{d}}\widehat{G}_{n}(t,\xi)\e^{in\cdot y}\,,\quad a_{n}(\xi;t):= \widehat{G}_{n}(t,\xi)\,.
\]
In what follows we forget the variable $\xi$, which is just parametric in the system \eqref{eq:s-eff}. By the definition \eqref{eq:eff} of the effective system, $a$ is solution to
\[
i\partial_{t}a_{n}(t) = \mathcal{F}\mathcal{R}[a(t)](n) + \mathcal{F}\mathcal{Q}^{t}[a(t)](n)\,,\quad n\in\Z^{d}\,.
\]
Therefore, for $\alpha\geq0$ we have
\[
\frac{1}{2}\frac{\mathrm{d}}{\mathrm{d}t}\|\pi_{\alpha}a\|_{\ell^{2}(\T^{d})}^{2} 
	= \im
	\big(
	\sum_{n\in\mathscr{C}_{\alpha}}
	\mathcal{F}\mathcal{R}[a(t)](n)\overline{a_{n}}(t) + \mathcal{F}\mathcal{Q}^{t}[a(t)](n)\overline{a_{n}}(t)
	\big)\,.
\]
According to Lemma \ref{lem:cons}, the quasi-resonant interactions $\mathcal{Q}^{t}$ do not contribute. We easily prove that the resonant interactions do not contribute either:
\[
\im
	\big(
	\sum_{n\in\mathscr{C}_{\alpha}}\mathcal{F}\mathcal{R}[a(t)](n)\overline{a_{n}}(t)
	\big)
	=
	\im
	\big(
	2\sum_{n\in\mathscr{C}_{\alpha}}\sum_{n_{1}\in\Z^{d}}|a_{n_{1}}(t)|^{2}|a_{n}(t)|^{2} - \sum_{n\in\mathscr{C}_{\alpha}}|a_{n}(t)|^{4}
	\big)
	=0\,.
\]
This proves \eqref{eq:cs}.  Then, the identities \eqref{eq:cs1} follow form Definition \ref{def:sob} of the $h^{s}$-norm based on the cluster decomposition: for all $\xi\in\R$,
\begin{multline*}
\|\widehat{G}(t,\xi)\|_{h^{s}(\T^{d})} 
	=  \big(\sum_{\alpha\in\N}K_{\alpha}^{2s}\|\pi_{\alpha}\widehat{G}(t,\xi)\|_{\ell^{2}(\T^{d})}^{2} \big)^{\frac{1}{2}} 
	\\
	=  \big(\sum_{\alpha\in\N}K_{\alpha}^{2s}\|\pi_{\alpha}\widehat{G}(1,\xi)\|_{\ell^{2}(\T^{d})}^{2} \big)^{\frac{1}{2}} 
	= \|\widehat{G}(1,\xi)\|_{h^{s}(\T^{d})} \,.
\end{multline*}
This completes the proof of Lemma \ref{lem:eff}.
 \end{proof}

\subsubsection{The time oscillations}
\label{sec:fast}
We turn to the contribution of $\mathcal{E}^{t}$ defined in \eqref{eq:Eo}. We first prove a trilinear estimate that overcomes the small-divisor problems arising in the Poincaré--Dulac normal form. This estimate is a consequence of Lemma \ref{lem:small-div}, which was itself a consequence of the frequency configuration of the interactions contained in $\mathcal{E}^{t}$. 
\begin{lemma}\label{lem:fast}
For $n\in\bigcup_{\alpha\geq\alpha_{0}}\mathscr{C}_{\alpha}$, we let 
\begin{equation}
\label{eq:tilde}
\mathcal{F}\widetilde{\mathcal{E}}^{t}[F^{a},F^{b},F^{c}](\xi,n):= 
	\sum_
	{\substack{(n_{1},n_{2},n_{3})\notin\Lambda^{(1)}(n)\cup\Lambda^{(3)}(n)\\n_{1}-n_{2}+n_{3}=n\\
	n_{1},n_{3}\neq n}}
	\frac{\e^{it\Omega(\vec{n})}}{i\Omega(\vec{n})}
	\widehat{F}_{n_{1}}^{a}(\xi)
	\overline{\widehat{F}_{n_{2}}^{b}}(\xi)
	\widehat{F}_{n_{3}}^{c}(\xi)\,.
\end{equation}
Then, for all $t>0$ and $\xi\in\R$,
\begin{align*}
	\|
	\widehat{\widetilde{\mathcal{E}}^{t}[F^{a},F^{b},F^{c}]}(\xi)
	\|_{h^{s}(\T^{d})} 
	&\lesssim
	\|\widehat{F^{a}}(\xi)\|_{h^{s}(\T^{d})}\|\widehat{F^{b}}(\xi)\|_{h^{s}(\T^{d})}\|\widehat{F^{c}}(\xi)\|_{h^{s}(\T^{d})}\,.
\intertext{In particular,}
	\|
	\widetilde{\mathcal{E}}^{t}[F^{a},F^{b},F^{c}]
	\|_{Z}
	&\lesssim
	\|F^{a}\|_{Z}\|F^{b}\|_{Z}\|F^{c}\|_{Z}\,.
	\\
	\|
	\widetilde{\mathcal{E}}^{t}[F^{a},F^{b},F^{c}]
	\|_{S} &\lesssim \sum_{\sigma\in\mathfrak{S}_{3}} \|F^{\sigma(a)}\|_{S}\|F^{\sigma(b)}\|_{Z}\|F^{\sigma(c)}\|_{Z}\,.
\end{align*}
\end{lemma}

\begin{proof} For fixed $\xi\in\R$, we have by duality that 
\[
	\|\widetilde{\mathcal{E}}^{t}[F^{a},F^{b},F^{c}](\xi)\|_{h^{s}(\T^{d})}
	\lesssim
	\sup_{\|v\|_{\ell^{2}}\leq1}
	\sum_{\alpha\geq\alpha_{0}}\sum_{n\in\mathscr{C}_{\alpha}}
	K_{\alpha}^{s}
	|
	\mathcal{F}
	\widetilde{\mathcal{E}}^{t}[F^{a},F^{b},F^{c}]
	(\xi,n)
	|
	|v_{n}|\,.
\]
According to Lemma \ref{lem:small-div}, we have that for all $n\in\bigcup_{\alpha\geq\alpha_{0}}\mathscr{C}_{\alpha}$ and
\[
(n_{1},n_{2},n_{3})~\in~(\Z^{d})^{3}\setminus(\Lambda^{(1)}(n)\cup\Lambda^{(3)}(n))
\]
with $n=n_{1}-n_{2}+n_{3}$, 
\[
\frac{1}{|\Omega(\vec{n})|}K_{\alpha}^{s} \lesssim (n_{1}^{\ast})^{s}(n_{2}^{\ast})^{\frac{4\tau}{c(d)}}\,.
\]
We deduce that 
\[
|\mathcal{F}\widetilde{\mathcal{E}}^{t}[F^{a},F^{b},F^{c}](\xi,n)|
	\lesssim
	\sum_{n_{1}-n_{2}+n_{3}=n} (n_{1}^{\ast})^{s}(n_{2}^{\ast})^{\frac{4\tau}{c(d)}} |\widehat{F}_{n_{1}}^{a}(\xi)\widehat{F}_{n_{2}}^{b}(\xi)\widehat{F}_{n_{3}}^{c}(\xi)|\,.
\]
We conclude from the Young's convolution inequality and the Sobolev embedding that for all $\xi\in\R$ and $\eta>0$,
\begin{multline*}
\|
	\mathcal{F}
	\widetilde{\mathcal{E}}^{t}[F^{a},F^{b},F^{c}](\xi)
\|_{h^{s}(\T^{d})}
	\lesssim
	\sum_{\sigma\in\mathfrak{S}_{3}}
	\|\langle n\rangle^{s} \widehat{F_{n}^{\sigma(a)}}(\xi)\|_{\ell_{n}^{2}}
	\|\langle n\rangle^{\frac{4\tau}{c(d)}}\widehat{F_{n}^{\sigma(b)}}(\xi)\|_{\ell_{n}^{1}}
	\|\widehat{F_{n}^{\sigma(c)}}(\xi)\|_{\ell_{n}^{1}}\,,\\
	\lesssim_{\eta}
	\sum_{\sigma\in\mathfrak{S}_{3}}
	\|\langle n\rangle^{s} \widehat{F_{n}^{\sigma(a)}}(\xi)\|_{\ell_{n}^{2}}
	\|\langle n\rangle^{\frac{4\tau}{c(d)}+\frac{d}{2}+\eta}\widehat{F_{n}^{\sigma(b)}}(\xi)\|_{\ell_{n}^{2}}
	\|\langle n\rangle^{\frac{d}{2}+\eta}\widehat{F_{n}^{\sigma(c)}}(\xi)\|_{\ell_{n}^{2}}\,,
\end{multline*}
which is conclusive under our assumption that $s>\frac{d}{2}+\frac{4\tau}{c(d)}$.
\end{proof}

\begin{lemma}[Poincaré--Dulac normal form]\label{lem:Poi-Dul} Suppose that $T\geq1$ and $F\in X_{T}$. Then, for all $t\in[1,T]$ and $\xi\in\R$,
\begin{equation}
	\label{eq:pde}
	\sum_{\alpha\in\N}K_{\alpha}^{2s}
	 \big|
	\sum_{n\in\mathscr{C}_{\alpha}}
	\int_{1}^{t}\mathcal{F}\mathcal{E}^{\tau}[F(\tau)](\xi,n)\overline{\widehat{F}_{n}}(\tau,\xi)
	\mathrm{d}\tau
	 \big|
	\lesssim \|F\|_{X_{T}}^{4}\,.
\end{equation}
Moreover, for all $t_{0}\in[\max(1,t/4),t]$,
\begin{equation}
		\label{eq:pds}
		\left\|
		\int_{t_{0}}^{t}\mathcal{E}^{\tau}[F(\tau)]\,\mathrm{d}\tau
		\right\|_{Z}
		+
		\left\|
		\int_{t_{0}}^{t}\mathcal{E}^{\tau}[F(\tau)]\,\mathrm{d}\tau
		\right\|_{H^{s}(\R\times\T^{d})}
		\lesssim t^{-1+4\delta}\|F\|_{X_{T}}^{3}\,.
\end{equation}

\end{lemma}

\begin{proof}

Recalling the expression \eqref{eq:Eo} of $\mathcal{E}^{t}$, we have that for fixed $\xi\in\R$ and $n~\in~\bigcup_{\alpha\geq\alpha_{0}}\mathscr{C}_{\alpha}$,
\begin{align*}
\mathcal{F}\mathcal{E}^{\tau}[F(\tau)](\xi,n) 
	&= 
		\frac{\pi}{\tau}\sum_{\substack{(n_{1},n_{2},n_{3})\notin(\Lambda^{(1)}(n)\cup\Lambda^{(3)}(n))\\n_{1}-n_{2}+n_{3}=n\\
		n_{1},n_{3}\neq n}}
	\e^{i\tau\Omega(\vec{n})}
	\widehat{F}_{n_{1}}(\tau,\xi)
	\overline{\widehat{F}_{n_{2}}}(\tau,\xi)
	\widehat{F}_{n_{3}}(\tau,\xi)\,.
\intertext{For $n\in\bigcup_{\alpha<\alpha_{0}}\mathscr{C}_{\alpha}$, }
\mathcal{F}\mathcal{E}^{\tau}[F(\tau)](\xi,n) 
	&=
	\frac{\pi}{\tau}\sum_{\substack{n_{1}-n_{2}+n_{3}=n\\\Omega(\vec{n})\neq0}}\e^{i\tau\Omega(\vec{n})}
	\widehat{F}_{n_{1}}(\tau,\xi)
	\overline{\widehat{F}_{n_{2}}}(\tau,\xi)
	\widehat{F}_{n_{3}}(\tau,\xi)\,.
\end{align*}
Since the terms contributing to $\mathcal{E}^{t}$ are not resonant, $\Omega(\vec{n})\neq0$ and the Leibniz rule gives 
\begin{align*}
\frac{\pi}{\tau}
	\e^{i\tau\Omega(\vec{n})}
	\widehat{F}_{n_{1}}(\tau,\xi)
	\overline{\widehat{F}_{n_{2}}}(\tau,\xi)
	\widehat{F}_{n_{3}}(\tau,\xi)
	=&
	\frac{\mathrm{d}}{\mathrm{d}\tau} \big(
	\frac{\pi}{\tau}
	\frac{\e^{i\tau\Omega(\vec{n})}}{i\Omega(\vec{n})}
	\widehat{F}_{n_{1}}(\tau,\xi)
	\overline{\widehat{F}_{n_{2}}}(\tau,\xi)
	\widehat{F}_{n_{3}}(\tau,\xi) 
	 \big)
	 \\
	&+\frac{\pi}{\tau^{2}}\frac{\e^{i\tau\Omega(\vec{n})}}{i\Omega(\vec{n})}
	\widehat{F}_{n_{1}}(\tau,\xi)
	\overline{\widehat{F}_{n_{2}}}(\tau,\xi)
	\widehat{F}_{n_{3}}(\tau,\xi) 
	 \\
	&- \frac{\pi}{\tau}\frac{\e^{i\tau\Omega(\vec{n})}}{i\Omega(\vec{n})}
		\frac{\mathrm{d}}{\mathrm{d}\tau}
	 \big(
	\widehat{F}_{n_{1}}(\tau,\xi)
	\overline{\widehat{F}_{n_{2}}}(\tau,\xi)
	\widehat{F}_{n_{3}}(\tau,\xi) 
	 \big)
	\,.
\end{align*}
Hence, we can gain decay in time at the price of a small divisor $|\Omega(\vec{n})|^{-1}$. [TODO] next sentences unclear, correctFor the contributions to low-frequencies $n\in\bigcup_{\alpha<\alpha_{0}}\mathscr{C}_{\alpha}$. This loss can be easily absorbed in the energy estimates~\eqref{eq:pde} and~\eqref{eq:pds}: indeed, we can bound 
\[
K_{\alpha}^{s} \leq K_{\alpha_{0}}^{s} := C(s,d,\mathrm{A})\,,
\]
and the loss $|\Omega(\vec{n})|^{-1}\lesssim (n_{1}^{\ast})^{2\tau}$ can be absorbed (since $s>2\tau$). 
\medskip

For this reason, we only give details for the terms contributing to the frequency indices $n\in\bigcup_{\alpha\geq\alpha_{0}}\mathscr{C}_{\alpha}$. We will exploit Lemma \ref{lem:fast}, in which we proved how to overcome the small divisor loss. Defining the trilinear interaction $\widetilde{\mathcal{E}}^{t}$ as in \eqref{eq:tilde}, the quantity we need to bound reads
\begin{multline}
\label{eq:dt}
	\frac{\mathrm{d}}{\mathrm{d}\tau}
	 \big(
	\frac{\pi}{\tau}\mathcal{F}\widetilde{\mathcal{E}}^{\tau}[F,F,F](\xi,n) 
	 \big) 
	+\frac{\pi}{\tau^{2}}
	\mathcal{F}\widetilde{\mathcal{E}}^{\tau}[F,F,F](\xi,n)
	\\
	- 
	\frac{\pi}{\tau}
	 \big(\mathcal{F}\widetilde{\mathcal{E}}^{\tau}[\partial_{\tau}F,F,F]
	+
	\mathcal{F}\widetilde{\mathcal{E}}^{\tau}[F,\partial_{\tau} F,F]
	+
	\mathcal{F}\widetilde{\mathcal{E}}^{\tau}[F,F,\partial_{\tau}F]
	 \big)(\xi,n)\,,
\end{multline}
where  $n\in\bigcup_{\alpha\geq\alpha_{0}}\mathscr{C}_{\alpha}$. 
\medskip

We first prove the energy estimate~\eqref{eq:pde}. We integrate by parts and control the first contribution in~\eqref{eq:dt}:
\begin{multline*}
	\int_{1}^{t}
	\frac{\mathrm{d}}{\mathrm{d}\tau}
	 \big(\frac{\pi}{\tau}
	\mathcal{F}
	\widetilde{\mathcal{E}}^{\tau}[F](\xi,n)
	 \big)
	\overline{\widehat{F}_{n}}(\xi,\tau)
	\mathrm{d}\tau
	=
	-\int_{1}^{t}
	\frac{\pi}{\tau}
	\mathcal{F}
	\widetilde{\mathcal{E}}^{\tau}[F](\xi,n)
	\partial_{\tau}\overline{\widehat{F}_{n}}(\xi,\tau)
	\mathrm{d}\tau \\
		+ \frac{\pi}{t}\mathcal{F}\widetilde{\mathcal{E}}^{t}[F](\xi,n)\overline{\widehat{F}_{n}}(\xi,t)
		- \pi\mathcal{F}\widetilde{\mathcal{E}}^{1}[F](\xi,n)\overline{\widehat{F}_{n}}(\xi,1)\,.
\end{multline*}

Then, we deduce from the definition of the $Z$-norm that 
\begin{multline*}
	\sum_{\alpha\in\N}
	K_{\alpha}^{2s}
	\big|
	\sum_{n\in\mathscr{C}_{\alpha}}\int_{1}^{t}
	\frac{\mathrm{d}}{\mathrm{d}\tau}
	 \big(\frac{\pi}{\tau}
	\mathcal{F}
	\widetilde{\mathcal{E}}^{\tau}[F](\xi,n)
	 \big)
	\overline{\widehat{F}_{n}}(\xi,\tau)
	\mathrm{d}\tau
	 \big|
	\\
	\lesssim
	\int_{1}^{t}
	\|
	\mathcal{F}
	\widetilde{\mathcal{E}}^{\tau}[F]
	\|_{Z}
	\|
	\partial_{\tau}\widehat{F}
	\|_{Z}
	\frac{d\tau}{\tau}
	+ 
	\|
	\mathcal{F}\widetilde{\mathcal{E}}^{t}[F]
	\|_{Z}
	\|
	\widehat{F}_{n}(t)
	\|_{Z}
	+
	\|
	\mathcal{F}\widetilde{\mathcal{E}}^{1}[F]
	\|_{Z}
	\|
	\widehat{F}_{n}(1)
	\|_{Z}
	\\
	\lesssim
	\int_{1}^{t}
	\|F(\tau)\|_{Z}^{3}\|\partial_{\tau}F(\tau)\|_{Z}\frac{d\tau}{\tau}
	+
	\|F(t)\|_{Z}^{4}+\|F(1)\|_{Z}^{4}
	\,. 
\end{multline*}
In the last estimate, we applied the trilinear bound obtained in Lemma \ref{lem:fast}. By definition \eqref{eq:X} of the $X_{T}$-norm we obtain that for all $t\in[1,T]$, 
\begin{multline*}
	\sum_{\alpha\in\N}
	K_{\alpha}^{2s}
 	\big|
	\int_{1}^{t}
	\frac{\mathrm{d}}{\mathrm{d}\tau}
	 \big(\frac{\pi}{\tau}
	\mathcal{F}
		\widetilde{\mathcal{E}}^{\tau}[F](\xi,n)
	 \big)
	\overline{\widehat{F}_{n}}(\xi,\tau)
	\mathrm{d}\tau
	 \big|
	\lesssim \|F\|_{X_{T}}^{4}
	 \big(
	1+\int_{1}^{t}\frac{d\tau}{\tau^{2(1-\delta)}}
	 \big)	
\lesssim \|F\|_{X_{T}}^{4}\,.
\end{multline*}
The remaining contributions satisfy the same estimate. We now prove \eqref{eq:pds}. Since $\|G\|_{H^{s}}\leq\|G\|_{S}$, it suffices to estimate the $Z$- and $S$-norms. We only treat the high frequencies $n\in\bigcup_{\alpha\geq\alpha_{0}}\mathscr{C}_{\alpha}$, since the low-frequency loss is absorbed as above. Fix~$t_{0}\in[\max(1,t/4),t]$. Integration by parts give
\begin{multline*}
\int_{t_{0}}^{t}\mathcal{E}^{\tau}[F]\,\mathrm{d}\tau
=
\frac{\pi}{t}\widetilde{\mathcal{E}}^{t}[F,F,F]
-\frac{\pi}{t_{0}}\widetilde{\mathcal{E}}^{t_{0}}[F,F,F]
\\
+\int_{t_{0}}^{t}
\Big(
\frac{\pi}{\tau}\widetilde{\mathcal{E}}^{\tau}[F,F,F]
-\pi\sum_{\mathrm{cyc}}\widetilde{\mathcal{E}}^{\tau}[\partial_{\tau}F,F,F]
\Big)
\frac{\mathrm{d}\tau}{\tau}\,,
\end{multline*}

where $\sum_{\mathrm{cyc}}$ runs over the three placements of $\partial_{\tau}F$. The last two bounds of Lemma~\ref{lem:fast} control both required norms. In particular,
\[
\|\widetilde{\mathcal{E}}^{\tau}[F^{a},F^{b},F^{c}]\|_{Z}
+
\|\widetilde{\mathcal{E}}^{\tau}[F^{a},F^{b},F^{c}]\|_{S}
\lesssim
\sum_{\rho\in\mathfrak{S}_{3}}
\|F^{\rho(a)}\|_{S}\|F^{\rho(b)}\|_{Z}\|F^{\rho(c)}\|_{Z}\,.
\]
Using \eqref{eq:X} and $t_{0}\geq t/4$, the boundary terms satisfy
\[
\frac{1}{t}\big(\|\widetilde{\mathcal{E}}^{t}[F]\|_{Z}
+\|\widetilde{\mathcal{E}}^{t}[F]\|_{S}\big)
+\frac{1}{t_{0}}\big(\|\widetilde{\mathcal{E}}^{t_{0}}[F]\|_{Z}
+\|\widetilde{\mathcal{E}}^{t_{0}}[F]\|_{S}\big)
\lesssim
t^{-1+\delta}\|F\|_{X_{T}}^{3}\,.
\]
The same bounds give
\begin{multline*}
\int_{t_{0}}^{t}
\Big\|
\frac{\pi}{\tau}\widetilde{\mathcal{E}}^{\tau}[F,F,F]
-\pi\sum_{\mathrm{cyc}}\widetilde{\mathcal{E}}^{\tau}[\partial_{\tau}F,F,F]
\Big\|_{Z}
\frac{\mathrm{d}\tau}{\tau}
\\
+\int_{t_{0}}^{t}
\Big\|
\frac{\pi}{\tau}\widetilde{\mathcal{E}}^{\tau}[F,F,F]
-\pi\sum_{\mathrm{cyc}}\widetilde{\mathcal{E}}^{\tau}[\partial_{\tau}F,F,F]
\Big\|_{S}
\frac{\mathrm{d}\tau}{\tau}
\\
\lesssim
\|F\|_{X_{T}}^{3}\int_{t_{0}}^{t}\tau^{-2+2\delta}\,\mathrm{d}\tau
\lesssim
t^{-1+2\delta}\|F\|_{X_{T}}^{3}\,.
\end{multline*}

Since $2\delta\leq4\delta$, this proves \eqref{eq:pds}. This concludes the proof of Lemma \ref{lem:Poi-Dul}.
\end{proof}

\subsection{The space oscillations}
\label{sec:snr}

We now control the contribution of the space non-resonant interactions $\mathcal{N}^{t}_{\R,\mathrm{nr}}$ defined in \eqref{eq:d-fo}. For this purpose, we adapt the argument from Kato--Pusateri \cite{KP11} (see also \cite[Lemma 3.10]{HPTV15}). Here, we do not use the oscillations in the transverse direction~$\T^{d}$. 
\begin{lemma}[Space oscillations]
\label{lem:so}
For all $T\geq1$, $t\in[1,T]$,  and any $\gamma\in(0,\frac{1}{4})$ we have
\begin{align}
\label{eq:sos}
	\|
	\mathcal{N}^{t}_{\R,\mathrm{nr}}[F^{a},F^{b},F^{c}]
	\|_{Z}
	&\lesssim_{\gamma} t^{-1-\gamma+3\delta}
	\|F^{a}\|_{X_{T}}\|F^{b}\|_{X_{T}}\|F^{c}\|_{X_{T}}\,.
\intertext{Moreover, for $\delta>0$ sufficiently small (depending on $s,\sigma,d$ and $\gamma$), we have} 
\label{eq:so-b}
	\|\mathcal{N}^{t}_{\R,\mathrm{nr}}[F^{a},F^{b},F^{c}]\|_{H^{s}(\R\times\T^{d})} 
	&\lesssim t^{-1-2\delta}\|F^{a}\|_{X_{T}}\|F^{b}\|_{X_{T}}\|F^{c}\|_{X_{T}}\,.
\end{align}
\end{lemma}
In both cases we obtain an integrable decay in time provided that~$3\delta<\gamma$. The second bound \eqref{eq:so-b} will be useful to prove the modified scattering result in the topology of $H^{s}(\R\times\T^{d})$, from which we deduce the stability of the $H^{s}(\R\times\T^{d})$-norm claimed in the main Theorem \ref{thm:main}.
\begin{proof}
We first prove \eqref{eq:sos}. For this purpose we show that for all $\gamma\in(0,\frac{1}{4})$,
\begin{equation}
\label{eq:so}
	\|
	\mathcal{N}^{t}_{\R,\mathrm{nr}}[F^{a},F^{b},F^{c}]
	\|_{Z}
	\lesssim_{\gamma} t^{-1-\gamma}
	\|F^{a}\|_{S}\|F^{b}\|_{S}\|F^{c}\|_{S}\,.
\end{equation}
Following \cite{KP11} we have that for all $\gamma\in(0,\frac{1}{4}]$, $t>0$, and~$f^{a},f^{b},f^{c}$, 
\begin{align*}
\sup_{\xi\in\R}\, |\mathcal{O}^{t}[f^{a},f^{b},f^{c}](\xi)| &\lesssim t^{-1-\gamma}
	\int_{\R^{3}}
	|\eta|^{\gamma}|\zeta|^{\gamma}
	|f^{a}(t,x-\eta)|
	|f^{b}(t,x)|
	|f^{c}(t,x-\zeta)|
	dxd\eta \mathrm{d}\zeta \\
	&\lesssim
	t^{-1-\gamma}
	\max_{\sigma\in\mathfrak{S}_{3}}
	\|f^{\sigma(a)}\|_{L_{x}^{1}(\R)}
	\||x|^{2\gamma} f^{\sigma(b)}\|_{L_{x}^{1}}
	\||x|^{2\gamma} f^{\sigma(c)}\|_{L_{x}^{1}} \\
	&\lesssim
	t^{-1-\gamma}
	\|\langle x\rangle^{\frac{1}{2}+2\gamma+0}f^{a}\|_{L_{x}^{2}}
	\|\langle x\rangle^{\frac{1}{2}+2\gamma+0}f^{b}\|_{L_{x}^{2}}
	\|\langle x\rangle^{\frac{1}{2}+2\gamma+0}f^{c}\|_{L_{x}^{2}}\,,
\end{align*}
where we denoted by $\mathfrak{S}_{3}$ the group of permutations of three elements. We obtain that for fixed $(n_{1},n_{2},n_{3})$ and $\gamma\in(0,\frac{1}{4})$, 
\[
\sup_{\xi\in\R}
	|\mathcal{O}^{t}[F_{n_{1}}^{a},F_{n_{2}}^{b},F_{n_{3}}^{c}](\xi)| 
	\lesssim_{\gamma} 
	t^{-1-\gamma}
	\mathbf{f}_{n_{1}}^{a}\mathbf{f}_{n_{2}}^{b}\mathbf{f}_{n_{3}}^{c}\,,
\]
where we denote for $n\in\Z^{d}$
\[
\mathbf{f}_{n}^{a}:= \|\langle x\rangle F_{n}^{a}\|_{L_{x}^{2}(\R)}\,,
	\quad 
	\mathbf{f}_{n}^{b}:= \|\langle x\rangle F_{n}^{b}\|_{L_{x}^{2}(\R)}\,,
	\quad
	\mathbf{f}_{n}^{c}:= \|\langle x\rangle F_{n}^{c}\|_{L_{x}^{2}(\R)}\,.
\]
Observe that, since $s>\frac{d}{2}$,
\begin{equation}
\label{eq:a}
\|\mathbf{f}_{n}^{a}\|_{\ell_{n}^{1}(\Z^{d})} \lesssim
	\|\langle n\rangle^{s}\,\mathbf{f}_{n}^{a}\|_{\ell_{n}^{2}(\Z^{d})}
	\lesssim
	\|\langle x\rangle F^{a}\|_{L^{2}h^{s}(\R\times\T^{d})}
	\lesssim\|F^{a}\|_{S}\,.
\end{equation}
We obtain 
\begin{align*}
\|\mathcal{N}_{\R,\mathrm{nr}}^{t}[F^{a},F^{b},F^{c}]\|_{Z}
	&=
	\sup_{\xi\in\R}
	\|
	\sum_{n_{1}-n_{2}+n_{3}=n}\e^{it\Omega(\vec{n})}\mathcal{O}^{t}
	[
	F_{n_{1}}^{a},F_{n_{2}}^{b},F_{n_{3}}^{c}
	](\xi)
	\|_{h^{s}(\Z^{d})} \\
	&\lesssim_{\gamma} t^{-1-\gamma}
	\|
	\sum_{n_{1}-n_{2}+n_{3}=n}
	\mathbf{f}_{n_{1}}^{a}
	\mathbf{f}_{n_{2}}^{b}
	\mathbf{f}_{n_{3}}^{c}
	\|_{h^{s}(\Z^{d})}\\
	&\lesssim_{\gamma} t^{-1-\gamma}\sum_{\sigma\in\mathfrak{S}_{3}}
	\|\mathbf{f}_{n}^{\sigma(a)}\|_{\ell^{1}_{n}(\Z^{d})}
	\|\mathbf{f}_{n}^{\sigma(b)}\|_{\ell^{1}_{n}(\Z^{d})}
	\|\langle n\rangle^{s}\,\mathbf{f}_{n}^{\sigma(c)}\|_{\ell^{2}(\Z^{d})}\,,
\end{align*}
and the observation \eqref{eq:a} concludes the proof of estimate \eqref{eq:so}. 
\medskip

We now prove the bound~\eqref{eq:so-b}. We interpolate between the slightly non-integrable\footnote{That is, with time decay $t^{-1+\delta}$ for some $\delta>0$.} $S$-norm estimate, previously obtained for both the full system $\mathcal{N}^{t}$ and the space-resonant system $\mathcal{N}_{\R,\mathrm{res}}^{t}$, and the integrable bound~\eqref{eq:so}. For any $\alpha\in(0,1)$,
\begin{align*}
\|\mathcal{N}_{\R,\mathrm{nr}}^{t}&[F^{a},F^{b},F^{c}]\|_{L_{x}^{2}h_{y}^{s}}
	=
	 \big(
	\int_{\R}
	\|\mathcal{F}\mathcal{N}_{\R,\mathrm{nr}}^{t}[F^{a},F^{b},F^{c}](\xi)\|_{h_{y}^{s}}^{2}
	\mathrm{d}\xi
	 \big)^{\frac{1}{2}}
	\\
	&\hspace{-35pt}\leq
	\sup_{\xi\in\R}
	\|\mathcal{F}\mathcal{N}_{\R,\mathrm{nr}}^{t}[F^{a},F^{b},F^{c}](\xi)\|_{h_{y}^{s}}^{\alpha}
	 \big(\int_{\R}
	\langle\xi\rangle^{-2\alpha}
	\langle \xi\rangle^{2\alpha}
	\|\mathcal{F}\mathcal{N}_{\R,\mathrm{nr}}^{t}[F^{a},F^{b},F^{c}](\xi)\|_{h_{y}^{s}}^{2(1-\alpha)}
	\mathrm{d}\xi \big)^{\frac{1}{2}}\,.
\intertext{By Hölder's inequality, and assuming that $\frac{\alpha}{1-\alpha}<\sigma$, we obtain}
	&\lesssim
	\|\mathcal{N}_{\R,\mathrm{nr}}^{t}[F^{a},F^{b},F^{c}]\|_{Z}^{\alpha}
	 \big(
	\int_{\R}
	\langle\xi\rangle^{\frac{2\alpha}{1-\alpha}}
	\|\mathcal{F}\mathcal{N}_{\R,\mathrm{nr}}^{t}[F^{a},F^{b},F^{c}](\xi)\|_{h_{y}^{s}}^{2}
	\mathrm{d}\xi \big)^{\frac{1-\alpha}{2}}\\
	&\lesssim
	\|\mathcal{N}_{\R,\mathrm{nr}}^{t}[F^{a},F^{b},F^{c}]\|_{Z}^{\alpha}
	\|
	(1-\partial_{x}^{2})^\frac{\sigma}{2}
	\mathcal{N}_{\R,\mathrm{nr}}^{t}[F^{a},F^{b},F^{c}]
	\|_{L_{x}^{2}h_{y}^{s}}^{1-\alpha}.
\end{align*}
Similarly, we have that for all $\alpha\in(0,1)$,
\begin{align*}
\|\mathcal{N}_{\R,\mathrm{nr}}^{t}&[F^{a},F^{b},F^{c}]\|_{H_{x}^{s}\ell_{y}^{2}}
	=
	 \big(
	\int_{\R}
	\langle\xi\rangle^{2s}
	\|\mathcal{F}\mathcal{N}_{\R,\mathrm{nr}}^{t}[F^{a},F^{b},F^{c}](\xi)\|_{\ell_{y}^{2}}^{2}
	\mathrm{d}\xi
	 \big)^{\frac{1}{2}}
	\\
	&\hspace{-40pt}\leq
	\sup_{\xi\in\R}
	\|\mathcal{F}\mathcal{N}_{\R,\mathrm{nr}}^{t}[F^{a},F^{b},F^{c}](\xi)\|_{\ell_{y}^{2}}^{\alpha}
	 \big(\int_{\R}
	\langle\xi\rangle^{-2\alpha}
	\langle \xi\rangle^{2(s+\alpha)}
	\|\mathcal{F}\mathcal{N}_{\R,\mathrm{nr}}^{t}[F^{a},F^{b},F^{c}](\xi)\|_{\ell_{y}^{2}}^{2(1-\alpha)}
	\mathrm{d}\xi \big)^{\frac{1}{2}}
\intertext{We deduce from Hölder's inequality in the integral over $\xi$ that}
	&\lesssim
	\|\mathcal{N}_{\R,\mathrm{nr}}^{t}[F^{a},F^{b},F^{c}]\|_{Z}^{\alpha}
	 \big(
	\int_{\R}
	\langle\xi\rangle^{\frac{2(s+\alpha)}{1-\alpha}}
	\|\mathcal{F}\mathcal{N}_{\R,\mathrm{nr}}^{t}[F^{a},F^{b},F^{c}](\xi)\|_{\ell_{y}^{2}}^{2}
	\mathrm{d}\xi \big)^{\frac{1-\alpha}{2}}.
	\end{align*}
By choosing $\alpha>0$ sufficiently small (only depending on $s$ and $\sigma$ from the definition of the $S$-norm) to ensure that 
\[
\frac{s+\alpha}{1-\alpha}<s+\sigma\,,
\]
we conclude that
\[
\|\mathcal{N}_{\R,\mathrm{nr}}^{t}[F^{a},F^{b},F^{c}]\|_{H_{x,y}^{s}}
	\lesssim
	\|\mathcal{N}_{\R,\mathrm{nr}}^{t}[F^{a},F^{b},F^{c}]\|_{Z}^{\alpha}\
	\|(1-\partial_{x}^{2})^\frac{\sigma}{2}\mathcal{N}_{\R,\mathrm{nr}}^{t}[F^{a},F^{b},F^{c}]\|_{H_{x,y}^{s}}^{1-\alpha}\,.
\]
The trilinear bounds \eqref{eq:sb3} and \eqref{eq:boS} give the weighted estimate required above:
\begin{multline*}
\|(1-\partial_{x}^{2})^{\frac{\sigma}{2}}
\mathcal{N}_{\R,\mathrm{nr}}^{t}[F^{a},F^{b},F^{c}]\|_{H_{x,y}^{s}}
\leq
\|\mathcal{N}_{\R,\mathrm{nr}}^{t}[F^{a},F^{b},F^{c}]\|_{S}
\\
\lesssim
\|\mathcal{N}^{t}[F^{a},F^{b},F^{c}]\|_{S}
+
\|\mathcal{N}_{\R,\mathrm{res}}^{t}[F^{a},F^{b},F^{c}]\|_{S}
\lesssim
\langle t\rangle^{-1+\delta}
\prod_{\nu\in\{a,b,c\}}\|F^{\nu}\|_{X_{T}}\,.
\end{multline*}

Hence, we deduce from the first bound \eqref{eq:so} (in the $Z$-norm) that for all $t\geq1$, 
\[
\|\mathcal{N}_{\R,\mathrm{nr}}^{t}[F^{a},F^{b},F^{c}]\|_{H_{x,y}^{s}}
	\lesssim
	t^{-1-\alpha\gamma+(1-\alpha)\delta}\|F^{a}\|_{X_{T}}\|F^{b}\|_{X_{T}}\|F^{c}\|_{X_{T}}\,,
\]
which is conclusive when 
\[
-\alpha\gamma + (1-\alpha)\delta <-2\delta\,, 
\] 
satisfied for $\delta>0$ small enough (for fixed $\alpha$, $\gamma$). This proves  Lemma~\ref{lem:so}.
\end{proof}
\section{A priori bounds and asymptotic behavior}
\label{sec:a-priori}
In this section we bring together the multilinear bounds collected in the previous sections and prove the main results, namely Theorem~\ref{thm:main} and Theorem~\ref{thm:ms}. 
\subsection{A priori bounds} We already obtained the slightly divergent bounds for the strong norm $S$ in Lemma \ref{lem:sb}. It remains to prove the a priori bounds for the weak norm $Z$.
\begin{proposition}[Energy estimate]\label{prop:energy} There exists $C>0$ such that for all $T\geq1$, if $F\in X_{T}$ is a profile solution to the interaction system \eqref{eq:N} on $[0,T]$ then
\[
\sup_{t\in[0,T]}\|F(t)\|_{Z}^{2}\leq \|F(1)\|_{Z}^{2}+ C\|F\|_{X_{T}}^{4}\,.
\]
\end{proposition}

\begin{proof}
Lemma~\ref{lem:sb} and the embedding of $S$ into $Z$ control the energy variation on the time interval~$[0,1]$. It therefore suffices to consider $t\in[1,T]$.
For fixed $\xi\in\R$, we have
\[
\|\widehat{F}(t,\xi)\|_{h^{s}}^{2}\leq \|\widehat{F}(1,\xi)\|_{h^{s}(\T^{d})}^{2}
	+
	\sum_{\alpha\in\N}K_{\alpha}^{2s}
	\big|
	\int_{1}^{t}\frac{\mathrm{d}}{\mathrm{d}\tau}
	\|\pi_{\alpha}\widehat{F}(\xi,\tau)\|_{\ell^{2}(\T^{d})}^{2}
	\mathrm{d}\tau
	 \big|\,.
\]
According to the decomposition \eqref{eq:nl} we have that for all $\alpha\in\N$, 
\begin{multline*}
\frac{1}{2}\frac{\mathrm{d}}{\mathrm{d}\tau}\|\pi_{\alpha}\widehat{F}(\tau,\xi)\|_{\ell^{2}(\Z^{d})}^{2}
	\\
	= 
	\mathrm{Im}
	\Big(
	\sum_{n\in\mathscr{C}_{\alpha}}
	\overline{\widehat{F}_{n}}(\tau,\xi)
	\big(
	\mathcal{F}\mathcal{N}_{\mathrm{eff}}^{\tau}[F](\xi,n)
	+ 
	\mathcal{F}\mathcal{E}^{\tau}[F](\xi,n)
	+
	\mathcal{F}\mathcal{N}_{\R,\mathrm{nr}}^{\tau}[F](\xi,n)
	 \big)
	 \Big)\,.
\end{multline*}
Following the proof of Lemma \ref{lem:eff}, we show that the effective system involving $\mathcal{N}_{\mathrm{eff}}^{\tau}$ does not contribute to the above energy estimate. This gives
\[
\frac{1}{2}\frac{\mathrm{d}}{\mathrm{d}\tau}\|\pi_{\alpha}\widehat{F}(\xi,\tau)\|_{\ell^{2}(\Z^{d})}^{2}
	= 
	\mathrm{Im}
	\Big(
	\sum_{n\in\mathscr{C}_{\alpha}}\overline{\widehat{F}_{n}}(\xi,\tau)
	 \big( 
	\mathcal{F}\mathcal{E}^{\tau}[F](\xi,n)
	+
	\mathcal{F}\mathcal{N}_{\R,\mathrm{nr}}^{\tau}[F](\xi,n)
	 \big)
	 \Big)\,,
\]
and we obtain
\begin{multline*}
\|\widehat{F}(t,\xi)\|_{h^{s}}^{2} \leq \|\widehat{F}(1,\xi)\|_{h^{s}}^{2}
	+
	\sum_{\alpha\in\N}K_{\alpha}^{2s}
	\big|
	\sum_{n\in\mathscr{C}_{\alpha}}\int_{1}^{t}\mathcal{F}\mathcal{E}^{\tau}[F](\xi,n)\overline{\widehat{F}_{n}}(\tau,\xi)\mathrm{d}\tau
	\big|
	\\+ \int_{1}^{t}\|\mathcal{N}_{\R,\mathrm{nr}}^{\tau}[F](\xi)\|_{h^{s}}\|\widehat{F}(\tau,\xi)\|_{h^{s}}\mathrm{d}\tau\,.
\end{multline*}
A Poincaré--Dulac normal form performed in Lemma \ref{lem:Poi-Dul} allowed us to control the contributions of $\mathcal{E}^{\tau}$: invoking the bound~\eqref{eq:pde} into the above estimate gives
\[
\|\widehat{F}(t,\xi)\|_{h^{s}}^{2} \leq \|\widehat{F}(1,\xi)\|_{h^{s}}^{2}
	+ C\|F\|_{X_{T}}^{4} + \int_{1}^{t}\|\mathcal{N}_{\R,\mathrm{nr}}^{\tau}[F](\xi)\|_{h^{s}}\|\widehat{F}(\tau,\xi)\|_{h^{s}}\mathrm{d}\tau\,.
\]
Then, we apply \eqref{eq:sos} from Lemma \ref{lem:so} to control the space non-resonant interactions:
\begin{multline*}
\|\widehat{F}(t,\xi)\|_{h^{s}}^{2} \leq \|\widehat{F}(1,\xi)\|_{h^{s}}^{2}
	+ C\|F\|_{X_{T}}^{4} 
	+ \|F\|_{X_{T}}^{3}\int_{1}^{t}\tau^{-1-\gamma+3\delta}\|F(\tau)\|_{Z}\mathrm{d}\tau\\
	\leq 
	\|\widehat{F}(1,\xi)\|_{h^{s}}^{2}
	+ C\|F\|_{X_{T}}^{4} \,,
\end{multline*}
provided $3\delta<\gamma$. This concludes the proof of Proposition \ref{prop:energy}.
\end{proof}

\begin{proposition}[A priori bounds]
\label{prop:a-priori}
There exists $C>0$ such that for all $\epsilon>0$ sufficiently small and $U(0)\in S$  with 
\begin{equation}
\label{eq:small}
\|U(0)\|_{S}\leq \epsilon\,,
\end{equation}
the solution $U$ to \eqref{eq:nls} is global in  $C(\R;S)$ and satisfies for all $T\geq0$, 
\begin{equation}
\label{eq:gb1}
\|F\|_{X_{T}} \leq C\epsilon\,,
\end{equation}
where $F(t):=\e^{-it\Delta_{\mathrm{A}}}U(t)$. 
\end{proposition}
\begin{proof}
By time reversibility, it suffices to prove the estimates for nonnegative times.
Let $T\geq0$ and $0\leq t\leq T$. It follows from  Lemma \ref{lem:sb} that when~$0\leq t\leq1$, 
\[
\|F(t)-F(0)\|_{S}\leq \sup_{\tau\in[0,t]}\|\mathcal{N}^{\tau}[F]\|_{S} \lesssim \|F\|_{X_{T}}^{3}\,.
\]
Recalling from Remark \ref{rem:n} that the $S$-norm controls the $Z$-norm, we deduce from the definition \eqref{eq:X} of the $X_{T}$-norm that it remains to show that for all $T\geq1$ and~$1\leq t\leq T$,
\begin{multline}
\label{eq:ab}
\|F(t)\|_{Z}^{2}+t^{-2\delta}\|F(t)\|_{S}^{2} + t^{2(1-\delta)}\|\partial_{t}F(t)\|_{S}^{2} \\
	\leq
	\|F(1)\|_{Z}^{2} 
	+ C\|F\|_{X_{T}}^{4}
	+ \|F\|_{X_{T}}^{6}\,.
\end{multline}
We proved in Proposition \ref{prop:energy} that 
\[
\|F(t)\|_{Z}^{2} \leq \|F(1)\|_{Z}^{2} + C\|F\|_{X_{T}}^{4}\,,
\]
In Lemma \ref{lem:sb} we proved that for all $t\in[1,T]$,
\[
\|\partial_{t}F(t)\|_{S}= \|\mathcal{N}^{t}[F]\|_{S} \lesssim t^{-1+\delta}\|F\|_{X_{T}}^{3}\,.
\]
Therefore,
\[
\|F(t)-F(1)\|_{S}\leq \int_{1}^{t}\|\mathcal{N}^{\tau}[F]\|_{S} \mathrm{d}\tau \lesssim t^{\delta}\|F\|_{X_{T}}^{3}\,.
\]
Together with the short-time estimate, \eqref{eq:ab} yields
\[
\|F\|_{X_{T}}^{2}\leq C\epsilon^{2}+C\|F\|_{X_{T}}^{4}+C\|F\|_{X_{T}}^{6}\,.
\]
A standard continuity argument, starting from the local solution and bootstrapping $\|F\|_{X_T}\leq 2C^{1/2}\epsilon$, gives \eqref{eq:gb1} uniformly in $T$. The local well-posedness criterion then extends the solution globally.
\end{proof}

\subsection{Proof of the main Theorems}

We have all the ingredients to prove our main Theorems. 

\begin{proof}[Proof that Theorem~\ref{thm:ms} $\implies$ Theorem~\ref{thm:main}.]
The global existence and the dispersive estimate (the second inequality in~\eqref{eq:r1}) follow from the a priori bound~\eqref{eq:gb1} established in Proposition~\ref{prop:a-priori} together with Lemma~\ref{lem:disp}. 

Let $G$ denote the solution to the effective system~\eqref{eq:s-eff} provided by Theorem~\ref{thm:ms}. The uniform bound on the $H^{s}$-norm of $U$ for positive times follows from the conservation of the $H^{s}$-norm of $G$ (Lemma~\ref{lem:eff}), the modified scattering result in $H^{s}$, and the bound $\|G(1)\|_{H^{s}}\lesssim\epsilon$. Time reversibility gives the same bound for negative times.
\end{proof}

It remains to prove Theorem~\ref{thm:ms}. The argument follows the same strategy as the proof of~\cite[Theorem~6.1]{HPTV15}. We include it here in our setting for completeness.

\begin{proof}[Proof of Theorem \ref{thm:ms}] Let $U(0)$ satisfy the smallness condition~\eqref{eq:small} in the $S$-norm, and let $U$ be the corresponding global solution to~\eqref{eq:nls} given by Proposition~\ref{prop:a-priori}.
We construct a state $G_{\infty}(1)$ such that $U$ scatters (in the modified sense) to the solution of~\eqref{eq:s-eff} with data $G(1)=G_{\infty}(1)$. We fix $\delta>0$ small enough that $5\delta<\gamma$ and $6\delta<1$.
\medskip

Let $n\in\N$ and $T_{n}:=\e^{n}$. We let $(G_{n})_{n}$ be the sequence of solutions in $C([1,\infty);S)$ to~\eqref{eq:s-eff} with Cauchy data
\[
G_{n}(T_{n}):= F(T_{n})\,,\quad \text{where}\quad F(t)=\e^{-it\Delta_{\mathrm{A}}}U(t)\,.
\]
According to the conservation law \eqref{eq:cs1}, for all $t\geq1$,
\begin{equation}
\label{eq:16}
\|G_{n}(t)\|_{Z} = \|G_{n}(T_{n})\|_{Z} = \|F(T_{n})\|_{Z} \lesssim \epsilon\,,
\end{equation}
where we used the global bound \eqref{eq:gb1} on the $X_{T}$-norm of $F$ given by Proposition~\ref{prop:a-priori}. In addition, we have from Lemma \ref{lem:tri-res} that for $t\geq T_{n}$,
\begin{equation}
\label{eq:15}
\|\partial_{t}G_{n}(t)\|_{H_{x,y}^{s}} \lesssim t^{-1}\|G_{n}(t)\|_{Z}^{2}\|G_{n}(t)\|_{H_{x,y}^{s}} \lesssim t^{-1}\epsilon^{2}\|G_{n}(t)\|_{H_{x,y}^{s}}\,.
\end{equation}
Hence, using also the control on $G_{n}(T_{n})=F(T_{n})$ in $H_{x,y}^{s}$ given by \eqref{eq:gb1}, we deduce from the Gronwall's inequality that for all $t\geq T_{n}$,
\begin{equation}
\label{eq:16b}
\|G_{n}(t)\|_{H_{x,y}^{s}} \lesssim \epsilon t^{\delta}\,.
\end{equation}

\noindent {\bf Step 1:} We prove that for all $T_{n}\leq t\leq 4T_{n}$, 
\begin{equation}
\label{eq:17z}
\|F(t)-G_{n}(t)\|_{Z} \lesssim T_{n}^{-2\delta}\epsilon^{3}\,.
\end{equation}
Using the equation \eqref{eq:N} (resp. \eqref{eq:s-eff}) solved by $F$ (resp. $G$) and the Duhamel's integral formula, we obtain
\begin{multline*}
F(t)-G_{n}(t) = -i\int_{T_{n}}^{t}(\mathcal{E}^{\tau}[F(\tau)] + \mathcal{N}^{\tau}_{\R,\mathrm{nr}}[F(\tau)])\mathrm{d}\tau \\
-i \int_{T_{n}}^{t}(\mathcal{N}_{\mathrm{eff}}^{\tau}[F(\tau)]-\mathcal{N}_{\mathrm{eff}}^{\tau}[G_{n}(\tau)])\mathrm{d}\tau\,.
\end{multline*}
To control the source term (namely the first integral term) we use Lemma~\ref{lem:Poi-Dul} and Lemma~\ref{lem:so} together with the a priori bound~\eqref{eq:gb1} on $F$: for~$T_{n}\leq t\leq 4T_{n}$,
\[
	\big\|
	\int_{T_{n}}^{t}(\mathcal{E}^{\tau}[F] + \mathcal{N}^{\tau}_{\R,\mathrm{nr}}[F])\mathrm{d}\tau
	\big\|_{Z}
	\lesssim T_{n}^{-1+4\delta}\epsilon^{3} + \epsilon^{3}\int_{T_{n}}^{t}\tau^{-1-2\delta}\,\mathrm{d}\tau \lesssim T_{n}^{-2\delta}\epsilon^{3}\,.
\]
Factorizing 
\[
\mathcal{N}_{\mathrm{eff}}^{t}[F]-\mathcal{N}_{\mathrm{eff}}^{t}[G_{n}] = \mathcal{N}_{\mathrm{eff}}^{t}[F-G_{n},F,F] + \mathcal{N}_{\mathrm{eff}}^{t}[G_{n},F-G_{n},F] + \mathcal{N}_{\mathrm{eff}}^{t}[G_{n},G_{n},F-G_{n}]\,,
\]
 we deduce from the a priori bounds \eqref{eq:gb1} and \eqref{eq:16b} that
\begin{multline*}
\|F(t)-G_{n}(t)\|_{Z} \lesssim T_{n}^{-2\delta}\epsilon^{3} + \int_{T_{n}}^{t}(\|G_{n}(\tau)\|_{Z}^{2} + \|F(\tau)\|_{Z}^{2})\|F(\tau)-G_{n}(\tau)\|_{Z}\frac{d\tau}{\tau}
	\\
	\lesssim T_{n}^{-2\delta}\epsilon^{3} + \epsilon^{2}\int_{T_{n}}^{t}\|F(\tau)-G_{n}(\tau)\|_{Z}\frac{d\tau}{\tau}\,.
\end{multline*}
Gronwall's inequality then gives \eqref{eq:17z}. 
\medskip

\noindent{\bf Step 2:} We prove that for all $T_{n}\leq t\leq 4T_{n}$, 
\begin{equation}
\label{eq:17}
\|F(t)-G_{n}(t)\|_{H_{x,y}^{s}} \lesssim \epsilon^{3}T_{n}^{-\delta}\,.
\end{equation}
We proceed similarly but we need to control an extra source term by using \eqref{eq:17z}:
\begin{multline*}
\|F(t)-G_{n}(t)\|_{H_{x,y}^{s}} \lesssim \epsilon^{3}T_{n}^{-2\delta} + 
	\epsilon^{2}
	\int_{T_{n}}^{t}
	\|F(\tau)-G_{n}(\tau)\|_{H_{x,y}^{s}}\frac{d\tau}{\tau}
	\\	
	+\epsilon\int_{T_{n}}^{t}(\|G_{n}(\tau)\|_{H_{x,y}^{s}}+\|F(\tau)\|_{H_{x,y}^{s}})\|F(\tau)-G_{n}(\tau)\|_{Z} \frac{d\tau}{\tau}\,.
\end{multline*}
The bound \eqref{eq:17} follows from the a priori bounds \eqref{eq:16b}, \eqref{eq:17z} and the Gronwall's inequality. 
\medskip

\noindent{\bf Step 3:}
The bounds \eqref{eq:17z} and \eqref{eq:17} imply that, if $\epsilon$ is small enough, then for all~$n\in\N$,
\[
\|G_{n+1}(1)-G_{n}(1)\|_{Z}
+
\|G_{n+1}(1)-G_{n}(1)\|_{H_{x,y}^{s}}
\lesssim \epsilon^{3}\e^{-\frac{n\delta}{2}}\,.
\]
Indeed, since $T_{n+1}\leq4T_n$, \eqref{eq:17z} and \eqref{eq:17} give
\[
\|G_{n}(T_{n+1})-G_{n+1}(T_{n+1})\|_{Z}
+
\|G_{n}(T_{n+1})-G_{n+1}(T_{n+1})\|_{H_{x,y}^{s}}
\lesssim \epsilon^{3}T_{n}^{-\delta}\,.
\]
Applying the $Z$- and $H^s$-stability bounds from Lemma~\ref{lem:tri-res} backward on $[1,T_{n+1}]$ yields, for $\epsilon$ sufficiently small,
\[
\|G_{n}(1)-G_{n+1}(1)\|_{Z}
+
\|G_{n}(1)-G_{n+1}(1)\|_{H_{x,y}^{s}}
\lesssim \epsilon^{3}T_{n}^{-\delta}\e^{C\epsilon^{2}\log(T_{n+1})}
\lesssim \epsilon^{3}\e^{-\frac{n\delta}{2}}\,.
\]
In particular, $(G_{n}(1))_{n}$ converges in both $Z$ and $H^{s}(\R\times\T^{d})$ to a limit $G_{\infty}(1)$ satisfying
\begin{align*}
\|G_{\infty}(1)\|_{Z}+\|G_{\infty}(1)\|_{H_{x,y}^{s}}&\lesssim \epsilon\,,\quad
\\
\|G_{n}(1)-G_{\infty}(1)\|_{Z}
+
\|G_{n}(1)-G_{\infty}(1)\|_{H_{x,y}^{s}}
&\lesssim \epsilon^{3}\e^{-\frac{\delta n}{2}}\,.
\end{align*}

Let $G\in C([1,\infty);H^s)$ be the solution with $G(1)=G_{\infty}(1)$. Another application of Gronwall's inequality gives
\[
\sup_{1\leq t\leq T_{n+2}}\|G(t)-G_{n}(t)\|_{H_{x,y}^{s}}
\lesssim \epsilon^{3}\e^{-\frac{\delta n}{4}}\,.
\]
We conclude from \eqref{eq:17} that 
\begin{multline*}
\sup_{T_{n}\leq t\leq T_{n+1}} \|F(t)-G(t)\|_{H_{x,y}^{s}} \\
\lesssim \sup_{T_{n}\leq t\leq T_{n+1}} \|F(t)-G_{n}(t)\|_{H_{x,y}^{s}} + \sup_{T_{n}\leq t\leq T_{n+1}} \|G(t)-G_{n}(t)\|_{H_{x,y}^{s}} 
	\lesssim \epsilon^{3}\e^{-\frac{\delta n}{4}}\,,
\end{multline*}
which completes the proof of Theorem \ref{thm:ms}.
\end{proof}

\appendix

\section{Some elementary Lemmas}
\label{app}

To keep the exposition self-contained, we recall the proof of the dispersive estimate for the free Schrödinger evolution on $\R$.

\begin{proof}[Proof of \eqref{eq:disp}] We have 
\begin{multline*}
\e^{it\partial_{x}^{2}}f(x) = \frac{1}{(4i\pi t)^{\frac{1}{2}}}
	\int_{\R}\e^{\frac{i(x-z)^{2}}{4t}}f(z)\mathrm{d}z
	=
	\frac{\e^{\frac{ix^{2}}{4t}}}{(4i\pi t)^{\frac{1}{2}}}
	\int_{\R}
	\e^{-\frac{ixz}{2t}}\e^{\frac{iz^{2}}{4t}}f(z)\mathrm{d}z
	\\
	=
	\left(\frac{\pi}{it}\right)^{\frac{1}{2}}\e^{\frac{ix^{2}}{4t}}
	 \big[
	\widehat{f}(\frac{x}{2t})+ r[f](t,x)
	 \big]
	\,,
\end{multline*}
where
\[
r[f](t,x):=\frac{1}{2\pi}\int_{\R}\e^{-\frac{ixz}{2t}}(\e^{\frac{iz^{2}}{4t}}-1)f(z)\mathrm{d}z\,.
\]
Proving \eqref{eq:disp} reduces to bounding $r[f]$. For this, we split the integral into two different regions: $r_{1}[f](t,x)$ (resp. $r_{2}[f](t,x)$) corresponds to the region $|z|\leq\sqrt{t}$ (resp. $|z|>\sqrt{t}$).  On the one hand, for $x\in\R$ and $|t|\geq1$,
\begin{multline*}
|r_{1}[f](t,x) |\lesssim \frac{1}{t}\int_{|z|\leq\sqrt{t}}z^{2}|f(z)|\mathrm{d}z \lesssim \frac{1}{t}\big(\int_{|z|\leq\sqrt{t}}z^{2}\mathrm{d}z\big)^{\frac{1}{2}}\|zf(z)\|_{L^{2}(\R)}
	\\
	\lesssim\frac{1}{t^{\frac{1}{4}}}\|zf(z)\|_{L^{2}(\R)}\,.
\end{multline*}
On the other hand, 
\begin{multline*}
|r_{2}[f](t,x)| \lesssim \int_{|z|\geq\sqrt{t}}|f(z)|\mathrm{d}z 
	\lesssim \big(\int_{|z|\geq\sqrt{t}}\frac{1}{|z|^{2}}\mathrm{d}z\big)^{\frac{1}{2}}
	\|zf(z)\|_{L^{2}(\R)} \\
	\lesssim\frac{1}{t^{\frac{1}{4}}}\|zf(z)\|_{L^{2}(\R)}\,.
\end{multline*}
This completes the proof of the dispersive estimate \eqref{eq:disp}.
\end{proof}
The product rule stated in the following lemma was useful in the trilinear estimates on the strong norm of Lemma \ref{lem:sb}.
\begin{lemma}[Product rule]\label{lem:prod} For all $s>\frac{d}{2}$ and $\sigma\geq0$, there exists $C>0$ such that for every functions $F,G\in (1-\partial_{x}^{2})^{-\sigma}H_{x,y}^{s}(\R\times\T^{d})\cap L_{x}^{\infty}h_{y}^{s}(\R\times\T^{d})$,
\begin{multline*}
\|(1-\partial_{x}^{2})^\frac{\sigma}{2}(FG)\|_{H_{x,y}^{s}(\R\times\T^{d})} 
	\lesssim 
	\|(1-\partial_{x}^{2})^{\frac{\sigma}{2}}F\|_{H_{x,y}^{s}(\R\times\T^{d})}
	\|G\|_{L_{x}^{\infty}h_{y}^{s}(\R\times\T^{d})}
	\\
	+ 
	\|
	F
	\|_{L_{x}^{\infty}h_{y}^{s}(\R\times\T^{d})}
	\|
	(1-\partial_{x}^{2})^\frac{\sigma}{2}G
	\|_{H_{x,y}^{s}(\R\times\T^{d})}\,.
\end{multline*}
\end{lemma}
\begin{proof}
The bound on the norm $H_{x}^{s}\ell_{y}^{2}$ follows from the product rule on $\R$:
\begin{align*}
\|FG\|_{H_{x}^{s}\ell_{y}^{2}} & \lesssim \| \|F\|_{H_{x}^{s}}\|G\|_{L_{x}^{\infty}} \|_{\ell_{y}^{2}} + \|\|F\|_{L_{x}^{\infty}}\|G\|_{H_{x}^{s}}\|_{\ell_{y}^{2}}  \\
	&\lesssim
	\|F\|_{H_{x}^{s}\ell_{y}^{2}}\|G\|_{L_{x}^{\infty}\ell_{y}^{\infty}}
	+
	\|F\|_{L_{x}^{\infty}\ell_{y}^{\infty}}\|G\|_{H_{x}^{s}\ell_{y}^{2}}\,.
\end{align*}
It remains to prove the bound for the norm $L_{x}^{2}h_{y}^{s}$.
\medskip

\noindent\textbullet\ \textbf{Case $\sigma=0$}: we have from the assumption $s>\frac{d}{2}$ that
\[
\|FG\|_{L_{x}^{2}h_{y}^{s}} \lesssim \|\|F\|_{h_{y}^{s}}\|G\|_{h_{y}^{s}}\|_{L_{x}^{2}} \lesssim \|F\|_{L_{x}^{\infty}h_{y}^{s}}\|G\|_{L_{x}^{2}h_{y}^{s}}\,.
 \]
\noindent\textbullet\ \textbf{Case $\sigma>0$}: we employ Littlewood--Paley decomposition in the $x$-variable only, and use that $h^{s}(\T^{d})$ is an algebra when $s>\frac{d}{2}$. Given $\varphi$ a bump function on $\R$, supported in $[-2,2]$ and equals to 1 on $[-1,1]$, we set for $N>1$ dyadic, $\xi\in\R$ and $y\in\T^{d}$
\[
\widehat{P_{N}F}(\xi,y) =  \big(\varphi(\frac{\xi}{N})-\varphi(\frac{\xi}{2N}) \big)\widehat{F}(\xi,y)\,,
\]
and
\[
\widehat{P_{\leq1}F}(\xi,y) = \varphi(\xi)\widehat{F}(\xi,y)\,.
\]
In particular, for all fixed $y\in\T^{d}$,
\[
\|(1-\partial_{x}^{2})^\frac{\sigma}{2}F(y)\|_{L_{x}^{2}(\R)} \sim 
	\big( \|P_{\leq1}F(y)\|_{L_{x}^{2}(\R)}^{2}+\sum_{N>1}N^{2\sigma}\|P_{N}F(y)\|_{L_{x}^{2}(\R)}^{2}
	\big)^{\frac{1}{2}}\,.
\]
In particular, we obtain from Fubini--Tonelli's theorem that 
\[
\|(1-\partial_{x}^{2})^\frac{\sigma}{2}F\|_{L_{x}^{2}h_{y}^{s}(\R\times\T^{d})} \sim 
	\big(
	\|P_{\leq1}F\|_{L_{x}^{2}h^{s}(\R\times\T^{d})}^{2}
	+
	\sum_{N>1}N^{2\sigma}
	\|
	P_{N}F
	\|_{L_{x}^{2}h_{y}^{s}(\R\times\T^{d})}^{2}
	\big)^{\frac{1}{2}}\,.
\]
For a fixed dyadic $N$ we decompose 
\[
\|P_{N}(FG)\|_{L_{x}^{2}h_{y}^{s}} \lesssim \|P_{<\frac{N}{8}}FP_{\sim N}G\|_{L_{x}^{2}h_{y}^{s}} 
	+ \sum_{M\geq\frac{N}{8}}
	\|P_{M}FG\|_{L_{x}^{2}h_{y}^{s}}\,,
\]
where we denoted 
\[
P_{\sim N} = \sum_{\frac{N}{8}<M<8N}P_{M}\,.
\]
To handle the first contribution, we observe that
\[
\|P_{<\frac{N}{8}}FP_{\sim N}G\|_{L_{x}^{2}h_{y}^{s}}
	\lesssim \|P_{<\frac{N}{8}}F\|_{L_{x}^{\infty}h_{y}^{s}}\|P_{\sim N}G\|_{L_{x}^{2}h_{y}^{s}}\,.
\]
Since
\[
 \|P_{<\frac{N}{8}}F\|_{L^{\infty}h^{s}} 
 	=
	\|
	\frac{1}{2\pi}\int_{\R}\widehat{\varphi}(x')F(\cdot + \frac{x'}{N})\mathrm{d}x'
	\|_{L_{x}^{\infty}h_{y}^{s}}
	\lesssim \|\widehat{\varphi}\|_{L^{1}(\R)}\|F\|_{L_{x}^{\infty}h_{y}^{s}}\,,
\]
we conclude that
\[
\|P_{<\frac{N}{8}}FP_{\sim N}G\|_{L_{x}^{2}h_{y}^{s}}
	\lesssim \|F\|_{L_{x}^{\infty}h_{y}^{s}}\|P_{\sim N}G\|_{L_{x}^{2}h_{y}^{s}}\,,
\]
which gives an acceptable contribution when summing over the frequency scales $N$. As for the second contribution, we write
\[
\sum_{M\geq\frac{N}{8}}\|P_{M}FG\|_{L^{2}h^{s}}
	\lesssim
	\big(\sum_{M\geq\frac{N}{8}}M^{-\sigma}\|(1-\partial_{x}^{2})^{\frac{\sigma}{2}}P_{M}F\|_{L_{x}^{2}h_{y}^{s}}
	\big)\|G\|_{L_{x}^{\infty}h_{y}^{s}}\,.
\]
Using Cauchy--Schwarz and the fact that $\sum_{M\geq\frac{N}{8}}(NM^{-1})^{\sigma}=\mathcal{O}(1)$ when $\sigma>0$ we obtain 
\begin{multline*}
N^{2\sigma}
	\big(
	\sum_{M\geq\frac{N}{8}}M^{-\sigma}\|(1-\partial_{x}^{2})^{\frac{\sigma}{2}}P_{M}F\|_{L_{x}^{2}h_{y}^{s}}
	\big)^{2}
	\\
	\lesssim	
	\sum_{M\geq\frac{N}{8}}(NM^{-1})^{\sigma}
	\sum_{M\geq\frac{N}{8}}(NM^{-1})^{\sigma}\|(1-\partial_{x}^{2})^{\frac{\sigma}{2}}P_{M}F\|_{L_{x}^{2}h_{y}^{s}}^{2}
	\\
	\lesssim
	\sum_{M\geq\frac{N}{8}}(NM^{-1})^{\sigma}\|(1-\partial_{x}^{2})^{\frac{\sigma}{2}}P_{M}F\|_{L_{x}^{2}h_{y}^{s}}^{2}\,.
\end{multline*}
Since $\sigma>0$ we can then sum over $N$ without losing derivatives for fixed $M$, and then sum over $M$ to conclude that this second contribution is also acceptable. 
\end{proof}

\bibliography{biblio}{}
	\bibliographystyle{alphaurl}

\end{document}